\documentclass[review,onefignum,onetabnum]{siamart171218}

\usepackage{braket,amsfonts}
\usepackage{yhmath}
\usepackage{amsmath}
\usepackage{mathtools}
\usepackage{mathdots}
\allowdisplaybreaks

\usepackage{array}

\usepackage[caption=false]{subfig}
\usepackage{pgfplots}

\usepackage{mathtools}
\newsiamthm{exmp}{Example}

\usepackage[T1]{fontenc}
\usepackage[utf8]{inputenc}

\usepackage{booktabs}

\usepackage{algorithmic}

\usepackage{graphicx,epstopdf}

\usepackage{amsopn}

\usepackage{xspace}
\usepackage{bold-extra}
\usepackage[most]{tcolorbox}

\colorlet{texcscolor}{blue!50!black}
\colorlet{texemcolor}{red!70!black}
\colorlet{texpreamble}{red!70!black}
\colorlet{codebackground}{black!25!white!25}

\lstdefinestyle{siamlatex}{%
  style=tcblatex,
  texcsstyle=*\color{texcscolor},
  texcsstyle=[2]\color{texemcolor},
  keywordstyle=[2]\color{texemcolor},
  moretexcs={cref,Cref,maketitle,mathcal,text,headers,email,url},
}

\tcbset{%
  colframe=black!75!white!75,
  coltitle=white,
  colback=codebackground, 
  colbacklower=white, 
  fonttitle=\bfseries,
  arc=0pt,outer arc=0pt,
  top=1pt,bottom=1pt,left=1mm,right=1mm,middle=1mm,boxsep=1mm,
  leftrule=0.3mm,rightrule=0.3mm,toprule=0.3mm,bottomrule=0.3mm,
  listing options={style=siamlatex}
}

\newtcblisting[use counter=example]{example}[2][]{%
  title={Example~\thetcbcounter: #2},#1}

\newtcbinputlisting[use counter=example]{\examplefile}[3][]{%
  title={Example~\thetcbcounter: #2},listing file={#3},#1}

\DeclareTotalTCBox{\code}{ v O{} }
{ 
  fontupper=\ttfamily\color{black},
  nobeforeafter,
  tcbox raise base,
  colback=codebackground,colframe=white,
  top=0pt,bottom=0pt,left=0mm,right=0mm,
  leftrule=0pt,rightrule=0pt,toprule=0mm,bottomrule=0mm,
  boxsep=0.5mm,
  #2}{#1}

\patchcmd\newpage{\vfil}{}{}{}
\numberwithin{equation}{section}

\newtheorem{remark}{Remark}[section]

\newcommand{\ds}{\mathrm ds}
\newcommand{\dx}{\mathrm dx}

\begin{tcbverbatimwrite}{tmp_\jobname_header.tex}
\title{Long time $H^1$-stability of fast L2-1$_\sigma$ method on general nonuniform meshes for subdiffusion equations\thanks{Submitted in Nov 2022.}}

\author{
Chaoyu Quan
\thanks{SUSTech International Center for Mathematics, Southern University of Science and Technology, Shenzhen, P.R. China (\email{quancy@sustech.edu.cn}).}
\and
Xu Wu
\thanks{Department of Mathematics,  Harbin Institute of Technology, Harbin 150001, China; Department of Mathematics, Southern University of Science and Technology, Shenzhen, China (\email{11849596@mail.sustech.edu.cn}).}
\and
Jiang Yang
\thanks{Department of Mathematics, SUSTech International Center for Mathematics \& National Center for Applied Mathematics Shenzhen (NCAMS), Guangdong Provincial Key Laboratory of Computational Science and Material Design, Southern University of Science and Technology, Shenzhen,China  (\email{yangj7@sustech.edu.cn}).}
}

\headers{Long time $H^1$-stability of fast L2-1$_\sigma$ method}{C. Quan, X. Wu and J. Yang}
\end{tcbverbatimwrite}
\title{Long time $H^1$-stability of fast L2-1$_\sigma$ method on general nonuniform meshes for subdiffusion equations\thanks{Submitted in Nov 2022.}}

\author{
Chaoyu Quan
\thanks{SUSTech International Center for Mathematics, Southern University of Science and Technology, Shenzhen, P.R. China (\email{quancy@sustech.edu.cn}).}
\and
Xu Wu
\thanks{Department of Mathematics,  Harbin Institute of Technology, Harbin 150001, China; Department of Mathematics, Southern University of Science and Technology, Shenzhen, China (\email{11849596@mail.sustech.edu.cn}).}
\and
Jiang Yang
\thanks{Department of Mathematics, SUSTech International Center for Mathematics \& National Center for Applied Mathematics Shenzhen (NCAMS), Guangdong Provincial Key Laboratory of Computational Science and Material Design, Southern University of Science and Technology, Shenzhen,China  (\email{yangj7@sustech.edu.cn}).}
}

\headers{Long time $H^1$-stability of fast L2-1$_\sigma$ method}{C. Quan, X. Wu and J. Yang}

\begin{document}
\maketitle
\sloppy

\begin{abstract}
In this work, the global-in-time $H^1$-stability of a fast L2-1$_\sigma$ method on general nonuniform meshes is studied for subdiffusion equations, where the convolution kernel in the Caputo fractional derivative is approximated by sum of exponentials. Under some mild restrictions on time stepsize, a bilinear form associated with the fast L2-1$_\sigma$  formula is proved to be positive semidefinite for all time. As a consequence,  the uniform global-in-time $H^1$-stability of the fast L2-1$_\sigma$ schemes can be  derived  for both linear  and semilinear  subdiffusion equations,  in the sense that the $H^1$-norm is uniformly bounded as the time tends to infinity. To the best of our knowledge, this appears to be  the first work for the global-in-time $H^1$-stability of fast L2-1$_\sigma$ scheme on general nonuniform meshes for subdiffusion equations. Moreover, the sharp finite time $H^1$-error estimate for the fast L2-1$_\sigma$ schemes is reproved based on more delicate analysis of coefficients where the restriction on time step ratios is relaxed comparing to existing works.
\end{abstract}


\section{Introduction}
In the past decade, the time-fractional diffusion equation \cite{metzler2000random,gorenflo2002time}  attract much attention from researchers in theoretical and numerical analysis.  Developing stable and accurate numerical method has been a core issue in this longstanding and active research area. Many well-known schemes have been developed and analyzed, including the piecewise polynomial interpolation methods (such as L1, L2-1$_\sigma$ and L2 schemes),  discontinuous Galerkin (dG) methods, and convolution quadrature (CQ) methods under the assumption that the solution is sufficiently smooth; see  \cite{langlands2005accuracy,sun2006fully, alikhanov2015new, gao2014new, lv2016error,jin2017correction,mustapha2014discontinuous}. 
However, the solutions to time-fractional problems typically admit weak singularities.  This inspires researchers to design  numerical schemes to overcome this singularity problem. For example, the L1, L2-1$_\sigma$,  L2 and dG methods on the graded meshes or general nonuniform meshes have been developed to overcome the singularities of time-fractional diffusion equation in \cite{stynes2017error, kopteva2019error, chen2019error, kopteva2020error, kopteva2021error, liao2018sharp, liao2019discrete,liao2018second, mustapha2014discontinuous, mustapha2012uniform}.  The L1, L2 and CQ schemes with proper initial correction using uniform step size can regain the desired high-order accuracy for linear subdiffusion problems with singularities; see \cite{yan2018analysis, xing2018higher,jin2017correction,jin2020subdiffusion}.  Recently, there are also plenty of works on the  stability and convergence of numerical methods for semilinear time-fractional equations; see \cite{jin2018numerical,al2019numerical,wang2020high,liao2020second,du2020time, li2022exponential}. Particularly, the energy dissipation property for a class of fractional phase field models has been a core issue \cite{liao2020second,du2020time,karaa2021positivity}, which essentially amounts to uniform global-in-time $H^1$-stability. This motivates us to study the long time $H^1$-stability of various numerical schemes for semilinear time-fractional equations.

In addition to the weak singularity problem the accuracy of numerical solutions,  another important issue  is the storage problem due to the nonlocality of the fractional derivatives.  Precisely speaking, all the aforementioned works require $\mathcal{O}(N)$ storage and $\mathcal{O}(N^2)$ computational cost for $N$ time steps, which are too expensive. Then the fast algorithms to reduce computational storage and cost have been developed.
Lubich-Sch{\"a}dle \cite{lubich2002fast} propose a new algorithm for the evaluation of convolution integrals 
for wave problems with non-reflecting boundary conditions. 
 Based on the block triangular Toeplitz matrix or block triangular Toeplitz-like matrix,
 Ke-Ng-Sun \cite{ke2015fast} and Lu-Peng-Sun \cite{lu2015fast} propose a fast approximation to time-fractional derivative. Baffet-Hesthaven \cite{baffet2017kernel} compress the kernel in the Laplace domain and obtained a sum-of-poles approximation for the Laplace transform of the kernel. Jiang-Zhang-Zhang-Zhang \cite{jiang2017fast} propose the sum-of-exponentials (SOE) approximation to speed up the evaluation of weakly singular  Caputo derivative kernel, which significantly reduces the computational storage and cost.  Along this way, the fast L1 and L2-1$_\sigma$ schemes are proposed and analyzed on both uniform and nonuniform meshes in \cite{liao2019unconditional,yan2017fast,li2021second,liu2022unconditionally, shen2018fast,song2021high}, and Zhu-Xu \cite{zhu2019fast} study the fast L2 method on uniform meshes based on the SOE approximation. 
 The SOE approximation allows us to reduce the storage and overall computational cost from  $\mathcal{O}(N)$  and $\mathcal{O}(N^2)$ for the L1, L2-1$_\sigma$, and L2 scheme to  $\mathcal{O}(N_q)$ and $\mathcal{O}(N_q N)$ for the fast method, respectively, with $N$ being the number of the time steps and $N_q$ being the number of the quadrature nodes. 

In this work, we apply the fast L2-1$_\sigma$ scheme on general nonuniform meshes for subdiffusion equation with homogeneous Dirichlet boundary condition, where the L2-1$_\sigma$  method \cite{alikhanov2015new} and the SOE  technique \cite{jiang2017fast, yan2017fast} are combined to approximate the Caputo fractional derivative.
Given a fractional order $\alpha\in(0,\,1)$, an absolute tolerance error $\varepsilon\ll{1}$, a cut-off time $\Delta{t}>0$ and  a fixed time $T_{\rm{soe}}>\Delta t$, it is well-known that there exists an SOE approximation to the convolution kernel, see Lemma \ref{lem:SOE_approx} for details. 
Denote the fast L2-1$_\sigma$ operator based on this SOE approximation by $F_k^{\alpha,*}$. In \cite{liao2020second}, the authors state that the positive semidefiniteness of the bilinear given as below form with the standard L2-1$_\sigma$ formula is an open problem, 
\begin{equation*}\label{eq:Bn}
    \mathcal B_n(v,w) = \sum_{k=1}^{n}\langle F_k^{\alpha,*} v, \delta_k w\rangle,\quad \delta_k w \coloneqq w^k-w^{k-1},~n\geq 1,
\end{equation*}
which is solved recently in \cite{quan2022stability} by some of us. 
In this work, we generalize this positive semidefiniteness to the fast L2-1$_\sigma$ formula. We prove that above bilinear form 
is positive semidefinite if the time steps $\{\tau_{k}\}_{k\geq 1}$ satisfy (see Theorem \ref{thm1})
\begin{equation*}
\begin{aligned}
    &\rho_k\coloneqq \tau_k/\tau_{k-1}\ge 0.475329,\quad
    \varepsilon\le \min _{k\ge 1}\frac{1}{5(1-\alpha)(\sigma\tau_k)^\alpha},\quad 
    \Delta t   \le \min_{k\ge 2}\sigma \tau_k, \\ 
    & T_{\rm{soe}}\ge \max_{k\ge2 }(\sigma\tau_{k+1}+\tau_k).
\end{aligned}
\end{equation*}
Based on this result, we prove that the $H^1$-stability of the fast  L2-1$_\sigma$ scheme for linear subdiffusion equation holds true even when the time $t$ is larger than $T_{\rm{soe}}$ and tends to infinity, (see Theorem \ref{thm2}),
while for semilinear subdiffusion equations, an additional restriction on $\tau_k$ should be satisfied to ensure similar stability (see Theorem \ref{thm:nonlinear}).  
To the best of our knowledge, this positive semidefiniteness result as well as the corresponding global-in-time $H^1$-stability is new.
On the other hand, the $H^1$-error estimate of the fast L2-1$_\sigma$ method have been well studied for subdiffusion equations on general nonuniform meshes in \cite{li2021second, liu2022unconditionally}.
But we still repeat the proof of the finite time optimal convergence rate in $H^1$-norm and reduce the restriction on time step ratios from $\rho_k\geq 2/3$ to $\rho_k\geq 0.475329 $.

This work is organized as follows. 
In Section~\ref{sect2}, the L2-1$_\sigma$ formula and the SOE approximation are recalled, and the fast L2-1$_\sigma$ formula is then derived. 
In Section~\ref{sect3}, we prove the positive semidefiniteness of the bilinear form $\mathcal B_n$ under some mild restrictions on the time steps. 
In Section~\ref{sect4}, we establish the global-in-time $H^1$-stability of the L2-1$_\sigma$ scheme for both the linear and semilinear subdiffusion equations, based on the above positive semidefiniteness result. In Section~\ref{sect5}, the sharp finite time $H^1$-error estimate of the fast L2-1$_\sigma$ scheme for the subdiffusion equation is provided. 
In Section~\ref{sect6}, we provide some numerical results to verify our theorems. 
Some conclusions are provided in Section~\ref{sect7}.

\section{ Fast L2-1$_\sigma$ formula}\label{sect2}
We first recall the L2-1$_\sigma$ method and then the fast L2-1$_\sigma$ method on an arbitrary nonuniform mesh for approximating the Caputo fractional derivative defined by
\begin{equation*}
    \partial_t^\alpha u = \frac{1}{\Gamma(1-\alpha)} \int_0^t \frac{u'(s)}{(t-s)^\alpha} \, \ds.
\end{equation*}
In the following content, we consider a nonuniform time mesh $0 = t_0<t_1<\ldots<t_k\ldots$ with time step $\tau_k = t_k-t_{k-1}$.

\subsection{L2-1$_\sigma$ formula}
We first briefly recall the L2-1$_\sigma$ scheme, which is constructed in \cite{alikhanov2015new}.
The fractional derivative $\partial_t^\alpha u(t_k^*)$ is approximated by the L2-1$_\sigma$ formula
\begin{align}\label{eq:lkstar}
    L_k^{\alpha,*} u 
    & = \frac{1}{\Gamma(1-\alpha)} \left(
    \sum_{j=1}^{k-1}\int_{t_{j-1}}^{t_j} \frac{\partial_s \Pi_{2,j} (s)}{(t_k^*-s)^\alpha}\,\ds + \int_{t_{k-1}}^{t_k^*} \frac{\partial_s \Pi_{1,k} (s)}{(t_k^*-s)^\alpha}\,\ds \right) \\
    & = \frac{1}{\Gamma(1-\alpha)}\left( \sum_{j=1}^{k-1} (a_{j}^{(k)} u^{j-1} + b_{j}^{(k)} u^j + c_{j}^{(k)} u^{j+1} ) 
    \right)+\frac{\sigma^{1-\alpha}(u^k-u^{k-1})}{\Gamma(2-\alpha)\tau_k^{\alpha}},\nonumber
\end{align} 
where $ t_k^*\coloneqq t_{k-1}+\sigma\tau_k$,  $\sigma=1-\alpha/2$, 
\begin{align*}
\Pi_{2,j} (t) &\coloneqq \frac{(t-t_j)(t-t_{j+1})}{(t_{j-1}-t_j)(t_{j-1}-t_{j+1})} u^{j-1} + \frac{(t-t_{j-1})(t-t_{j+1})}{(t_{j}-t_{j-1})(t_{j}-t_{j+1})} u^{j} \\ 
&\quad + \frac{(t-t_{j-1})(t-t_{j})}{(t_{j+1}-t_{j-1})(t_{j+1}-t_{j})} u^{j+1},\\
\Pi_{1,k} (t) & \coloneqq\frac{t-t_k}{t_{k-1}-t_k} u^{k-1} + \frac{t-t_{k-1}}{t_{k}-t_{k-1}} u^k,
\end{align*}
for $1\leq j\leq k-1$, and 
{\small
\begin{equation}\label{eq:akjl21sigma}
\begin{aligned}
   a_{j}^{(k)} & =  \int_{t_{j-1}}^{t_j} \frac{2s -t_j-t_{j+1}}{\tau_{j}(\tau_{j}+\tau_{j+1})} \frac{1}{(t_k^*-s)^\alpha}\,\ds 
  ,\quad
    b_{j}^{(k)}   = \int_{t_{j-1}}^{t_j} \frac{2s -t_{j-1}-t_{j+1}}{-\tau_{j}\tau_{j+1}} \frac{1}{(t_k^*-s)^\alpha}\,\ds 
  , \\
    c_{j}^{(k)}  & =  \int_{t_{j-1}}^{t_j} \frac{2s -t_{j-1}-t_{j}}{\tau_{j+1}(\tau_{j}+\tau_{j+1})} \frac{1}{(t_k^*-s)^\alpha}\,\ds.
\end{aligned}
\end{equation}}
The coefficients satisfy that  
$
a_{j}^{(k)}+b_{j}^{(k)}+c_{j}^{(k)} =0,
$
for $1\leq j\leq k-1$.
Then  the discrete fractional derivative in \eqref{eq:lkstar} can be reformulated as
{\small
\begin{equation}\label{eq:L21operator}
\begin{aligned}
L_k^{\alpha,*} u
  &=\frac{1}{\Gamma(1-\alpha)}\left( c_{k-1}^{(k)} \delta_k u -a_{1}^{(k)}\delta_1 u  +\sum_{j=2}^{k-1} d_{j}^{(k)} \delta_j u \right)+\frac{\sigma^{1-\alpha}}{\Gamma(2-\alpha)\tau_k^{\alpha}}\delta_k u, 
    \end{aligned} 
\end{equation}}
where $\delta_j u= u^j-u^{j-1}$ and $d_{j}^{(k)}\coloneqq c_{j-1}^{(k)}-a_{j}^{(k)}$. Here we make a convention that $a_{1}^{(1)}=0$ and $c_{0}^{(1)}=0$; see \cite{quan2022stability} for details.
\subsection{Fast L2-1$_\sigma$ formula} Note that the right-hand side of \eqref{eq:L21operator} involves a sum of all previous solutions $\{u^j\}_{j=0}^k$, which reflects the memory effect of the nonlocal fractional operator. This leads to expensive computational cost as the number of time steps increases. To improve the efficiency, one can use the following  SOE approximation result:
\begin{lemma}[\cite{jiang2017fast,liao2019unconditional}]\label{lem:SOE_approx}
For the given $\alpha\in(0,\,1)$, an absolute tolerance error $\varepsilon\ll{1}$, a cut-off time $\Delta{t}>0$ and  a fixed time $T_{\rm{soe}}>\Delta t$, there exists a positive integer $N_{q}$, positive quadrature nodes $\theta^{\ell}$ and corresponding positive weights $\varpi^{\ell}\,(1\leq{\ell}\leq{N_{q}})$ such that
\begin{align*}
\bigg|
t^{-\alpha}
-\sum_{\ell=1}^{N_{q}}\varpi^{\ell}e^{-\theta^{\ell}t}\bigg|\leq\varepsilon,
\quad
\forall\,{t}\in[\Delta{t},\,T_{\rm{soe}}],
\end{align*}
where the number $N_q$ of quadrature nodes satisfies
\begin{equation*}
N_q= \mathcal{O}\bigg(
\log\frac{1}{\varepsilon}\left(
\log\log\frac{1}{\varepsilon} + \log\frac{T_{\rm{soe}}}{\Delta t}\right)
+ \log\frac{1}{\Delta t}\left(
\log\log\frac{1}{\varepsilon} + \log\frac{1}{\Delta t}\right)
\bigg).
\end{equation*}
\end{lemma}
\begin{remark}
    Note that in this work, $T_{\rm{soe}}$ is not necessarily the total time. In the later $H^1$-stability analysis, the total time can go to infinity for fixed $T_{\rm{soe}}$.
\end{remark}
The L2-1$_\sigma$ formula $L_k^{\alpha,*}u$ is split  into a history part $H(t_{k}^*)$  and  a local part $L(t_k^*)$  as:
\begin{equation*}
   H(t_{k}^*)\coloneqq \frac{1}{\Gamma(1-\alpha)}\int_{0}^{t_{k-1}}\frac{\partial_s\Pi_2 (s)}{(t_k^*-s)^\alpha}\ \ds,\quad  L(t_k^*)\coloneqq \frac{1}{\Gamma(1-\alpha)}\int_{t_{k-1}}^{t_k^*}\frac{\partial_s \Pi_{1,k} (s)}{(t_k^*-s)^\alpha}\ \ds,
\end{equation*}
where $\Pi_2(s)=\Pi_{2,j}(s)$ on the interval $(t_{j-1},t_j)$, $1\le j\le k-1$. 
The fast L2-1$_\sigma$ formula is obtained by replacing $(t_k^*-s)^{-\alpha}$ in $H(t_k^*)$ with the SOE approximation in Lemma~\ref{lem:SOE_approx}:
\begin{align}\label{eq:soel21}
    F_k^{\alpha,*} u 
    & = \frac{1}{\Gamma(1-\alpha)} \left(
   \int_{0}^{t_{k-1}}\partial_s\Pi_2 (s)\sum_{\ell=1}^{N_{q}}\varpi^{\ell}e^{-\theta^{\ell}(t_k^*-s)}\,\ds + \int_{t_{k-1}}^{t_k^*} \frac{\partial_s \Pi_{1,k} (s)}{(t_k^*-s)^\alpha}\,\ds \right)\\
    & = \frac{  \sum_{\ell=1}^{N_{q}}\varpi^{\ell} H^\ell(t_{k}^*)}{\Gamma(1-\alpha)}+\frac{\sigma^{1-\alpha}(u^k-u^{k-1})}{\Gamma(2-\alpha)\tau_k^{\alpha}}, \nonumber
\end{align}   
where
$
    H^\ell(t_{k}^*)=\int_{0}^{t_{k-1}}\partial_s\Pi_2 (s)e^{-\theta^{\ell}(t_k^*-s)}\,\ds  \ \text{for} \ 1\le \ell\le N_q.
$ 
$H^\ell(t_k^*)$  can be computed by the following recurrence formulation:
\begin{equation*}\label{eq:recH}
\begin{aligned}
     H^\ell(t_{k}^*)&=\int_{0}^{t_{k-2}}\partial_s\Pi_2 (s)e^{-\theta^{\ell}(t_k^*-s)}\,\ds +\int_{t_{k-2}}^{t_{k-1}}\partial_s\Pi_{2,k-1} (s)e^{-\theta^{\ell}(t_k^*-s)}\,\ds\\
     &=e^{-\theta^{\ell}[(1-\sigma)\tau_{k-1}+\sigma\tau _k]}H^\ell(t_{k-1}^*)+\hat{a}^{k,\ell}_{k-1}u^{k-2}+\hat{b}^{k,\ell}_{k-1}u^{k-1}+\hat{c}^{k,\ell}_{k-1}u^{k},
\end{aligned}
\end{equation*}
with $ H^\ell(t_{1}^*)=0$ and for  $k\ge 2$
{ \small
\begin{align}\label{eq:akjfastl21sigma}
   \hat{a}^{k,\ell}_{k-1} & =  \int_{t_{k-2}}^{t_{k-1}} \frac{2s -t_{k-1}-t_{k}}{\tau_{k-1}(\tau_{k-1}+\tau_{k})} e^{-\theta^{\ell}(t_k^*-s)}\,\ds ,\quad
   \hat{b}^{k,\ell}_{k-1}  = \int_{t_{k-2}}^{t_{k-1}} \frac{2s -t_{k-2}-t_{k}}{-\tau_{k-1}\tau_{k}}e^{-\theta^{\ell}(t_k^*-s)}\,\ds ,\\
  \hat{c}^{k,\ell}_{k-1} & =  \int_{t_{k-2}}^{t_{k-1}} \frac{2s -t_{k-2}-t_{k-1}}{\tau_{k}(\tau_{k-1}+\tau_{k})}e^{-\theta^{\ell}(t_k^*-s)}\,\ds.\nonumber
\end{align}}
 Comparing $L^{\alpha,*}_k u$ with  $F_k^{\alpha,*} u$, we see that the former requires all the previous time step value $\{u^j\}_{j=1}^{k}$, while the latter only needs $u^k,\ u^{k-1},\ u^{k-2}$, and $H^\ell(t_k^*)$, $\ell=1,\cdots,N_q$. This means that  approximating $\partial_t^\alpha u(t_k^*)$ by $F_k^{\alpha,*} u$ rather than $L_k^{\alpha,*} u$  could reduces the storage and computational cost,  when $k$ is large. Roughly speaking, the storage cost can be reduced  from $\mathcal{O}(N)$ to $\mathcal{O}(N_q)$ and the computational cost can be reduced from $\mathcal{O}(N^2)$ to $\mathcal{O}(N_q N)$ by replacing $L_k^{\alpha,*} u$ with $F_k^{\alpha,*} u$.

The main purpose of this work is to establish global-in-time  $H^1$ stability of the  fast L2-1$_\sigma$ scheme for subdiffusion equation. 
\section{Positive semidefiniteness of bilinear form $\mathcal B_n$}\label{sect3}
 In this part,  we state and prove the main results on positive semidefiniteness of a bilinear form  associated with $F_k^{\alpha,*}u$, that  will be used to  establish the  global-in-time $H^1$-stability of  the fast L2-1$_\sigma$ scheme for subdiffusion equations.
 
 Firstly, we shall reformulate the formula \eqref{eq:soel21} as
 {\small
\begin{align}\label{eq:soeFk}
    F_k^{\alpha,*} u 
    & = \frac{1}{\Gamma(1-\alpha)} \left(\sum_{j=1}^{k-1}
   \int_{t_{j-1}}^{t_{j}}\partial_s\Pi_{2,j} (s)\sum_{\ell=1}^{N_{q}}\varpi^{\ell}e^{-\theta^{\ell}(t_k^*-s)}\,\ds + \int_{t_{k-1}}^{t_k^*} \frac{\partial_s \Pi_{1,k} (s)}{(t_k^*-s)^\alpha}\,\ds \right)\\
    & =  \frac{1}{\Gamma(1-\alpha)} \sum_{j=1}^{k-1} \left(\hat{a}_{j}^{(k)} u^{j-1} + \hat{b}_{j}^{(k)} u^j + \hat{c}_{j}^{(k)} u^{j+1}\right ) +
   \frac{\sigma^{1-\alpha}}{\Gamma(2-\alpha)\tau_k^{\alpha}}\delta_k u\nonumber\\
     &=\frac{1}{\Gamma(1-\alpha)} \sum_{j=1}^{k-1} \left(\hat{c}_{j}^{(k)} \delta_{j+1} u - \hat{a}_{j}^{(k)}\delta_j u\right)+\frac{\sigma^{1-\alpha}}{\Gamma(2-\alpha)\tau_k^{\alpha}}\delta_k u\nonumber\\
  &=\frac{1}{\Gamma(1-\alpha)}\left( \hat{c}_{k-1}^{(k)} \delta_k u   +\sum_{j=1}^{k-1} \hat{d}_{j}^{(k)} \delta_j u \right)+\frac{\sigma^{1-\alpha}}{\Gamma(2-\alpha)\tau_k^{\alpha}}\delta_k u, \nonumber
\end{align}  }
where  $\delta_j u= u^j-u^{j-1}$,
  \begin{align}\label{eq:aj}
   \hat{a}_{j}^{(k)} &  =  \int_{t_{j-1}}^{t_j} \frac{2s -t_j-t_{j+1}}{\tau_{j}(\tau_{j}+\tau_{j+1})}\sum_{\ell=1}^{N_{q}}\varpi^{\ell}e^{-\theta^{\ell}(t_k^*-s)}\,\ds,\nonumber\\
    \hat{b}_{j}^{(k)} & = \int_{t_{j-1}}^{t_j} \frac{2s -t_{j-1}-t_{j+1}}{-\tau_{j}\tau_{j+1}} \sum_{\ell=1}^{N_{q}}\varpi^{\ell}e^{-\theta^{\ell}(t_k^*-s)}\,\ds,\\
\hat{c}_{j}^{(k)}  & =\int_{t_{j-1}}^{t_j} \frac{2s -t_{j-1}-t_{j}}{\tau_{j+1}(\tau_{j}+\tau_{j+1})}\sum_{\ell=1}^{N_{q}}\varpi^{\ell}e^{-\theta^{\ell}(t_k^*-s)}\,\ds,\nonumber\\
\hat{d}^{(k)}_j&= \hat{c}^{(k)}_{j-1}-\hat{a}^{(k)}_{j},\nonumber
\end{align}
for $1\le j\le k-1$. In the reformulation \eqref{eq:soeFk}, we use the  fact $\hat{a}^{(k)}_{j}+\hat{b}^{(k)}_{j}+\hat{c}^{(k)}_{j}=0$. Here we make a convention that $\hat{c}^{(k)}_{0}=0$, i.e., $ \hat{d}^{(k)}_{1}=-\hat{a}^{(k)}_{1}$.

We now give some properties of the fast L2-1$_\sigma$ coefficients  $\hat{d}^{(k)}_j$ in \eqref{eq:aj}.
\begin{lemma}[Properties of  $\hat{d}^{(k)}_j$]\label{lemmapro}
For the fast L2-1$_\sigma$ coefficients $\hat{d}^{(k)}_j$  on a  nonuniform mesh $\{\tau_j\}_{j\ge1}$  defined in \eqref{eq:aj}, 
the following properties hold:
\begin{itemize}
\item[(P1)] $\hat{d}^{(k)}_j>0,~1\leq j\leq k-1,~k\geq 2$;
\item[(P2)] $\hat{d}^{(k+1)}_{j}-\hat{d}^{(k)}_j<0,~1\leq j\leq k-1,~k\geq 2$.
\end{itemize}
Furthermore, if the nonuniform mesh $\{\tau_j\}_{j\ge1}$
 with  $\rho_j \coloneqq \tau_j/\tau_{j-1}$ satisfies
\begin{equation}\label{condition:rho}
    \rho_j\ge \eta \approx 0.475329 ,\quad \forall j\geq 2,
\end{equation}
where $\eta $ is the unique positive root of $1-3\rho^2(1+\rho)=0$, then
\begin{itemize}
\item[(P3)] $\hat{d}^{(k)}_{j+1}-\hat{d}^{(k)}_j>0,~1\leq j\leq k-2,~k\geq 3$;
\item[(P4)] $\hat{d}^{(k)}_{j+1}-\hat{d}^{(k)}_j>\hat{d}^{(k+1)}_{j+1}-\hat{d}^{(k+1)}_{j},~1\leq j\leq k-2,~k\geq 3$.
\end{itemize}
\end{lemma}
\begin{proof}
 We first provide two equivalent reformulations of $\hat{a}^{(k)}_{j}$ according to \eqref{eq:aj}: $\forall 1\leq j\leq k-1$,
 {\small
\begin{align}\label{eq:akj_rem}
        \hat{a}^{(k)}_{j}
        &= \sum_{\ell=1}^{N_{q}} \int_0^1 \frac{-2 \tau_j(1-s)-\tau_{j+1}}{\tau_{j}+\tau_{j+1}}\varpi^{\ell}e^{-\theta^{\ell}(t_k^*-(t_{j-1}+s \tau_j))}\,\ds\\
        &=\sum_{\ell=1}^{N_{q}}\frac{1}{\tau_{j}+\tau_{j+1}}\int_0^1 \varpi^{\ell}e^{-\theta^{\ell}(t_k^*-(t_{j-1}+s \tau_j))}\,{\rm d}( \tau_j s^2-(2\tau_j+\tau_{j+1})s)\nonumber\\
     &=\sum_{\ell=1}^{N_{q}} \varpi^{\ell}\left(-e^{-\theta^{\ell}(t_k^*-t_{j})}+\frac{\theta^{\ell}\tau_j}{\tau_j+\tau_{j+1}}
    \int_0^{1} (\tau_j+\tau_{j+1}+s\tau_j)
    (1-s)e^{-\theta^{\ell}(t_k^*-t_{j}+s \tau_j)}\,\ds\right)\nonumber
    \end{align}}
and
{\small
    \begin{align}\label{eq:akj1_rem}
        \hat{a}^{(k)}_{j}  &=\sum_{\ell=1}^{N_{q}} \int_0^1 \frac{-2 \tau_j(1-s)-\tau_{j+1}}{\tau_{j}+\tau_{j+1}}\varpi^{\ell}e^{-\theta^{\ell}(t_k^*-(t_{j-1}+s \tau_j))}\,\ds\\
        &=\sum_{\ell=1}^{N_{q}} \int_0^1 \frac{-2 \tau_j s-\tau_{j+1}}{(\tau_{j}+\tau_{j+1})}\varpi^{\ell}e^{-\theta^{\ell}(t_k^*-t_{j}+s \tau_j)}\,\ds\nonumber\\
        &=\sum_{\ell=1}^{N_{q}}\frac{1}{\tau_{j}+\tau_{j+1}}\int_0^1 \varpi^{\ell}e^{-\theta^{\ell}(t_k^*-t_j+s \tau_j)}\,{\rm d}( - \tau_j s^2-\tau_{j+1}s)\nonumber\\
    &=-\sum_{\ell=1}^{N_{q}} \varpi^{\ell}\left(e^{-\theta^{\ell}(t_k^*-t_{j-1})}+\frac{\theta^{\ell}\tau_j}{\tau_j+\tau_{j+1}}
    \int_0^{1}(\tau_j+\tau_{j+1}-s\tau_j)(1-s)
    e^{-\theta^{\ell}(t_k^*-t_{j-1}-s \tau_j)}\,\ds\right).\nonumber
    \end{align}}
    Furthermore, we also reformulate 
$\hat{c}^{(k)}_{j}$ in \eqref{eq:aj} as: $\forall 1\leq j\leq k-1$, 
\begin{align}\label{eq:ckj_rem}
   \hat{c}^{(k)}_{j}   & 
 =  \sum_{\ell=1}^{N_{q}}\int_0^1 \frac{\tau_j^2(2s-1)}{\tau_{j+1}(\tau_{j}+\tau_{j+1})}\varpi^{\ell}e^{-\theta^{\ell}(t_k^*-(t_{j-1}+s \tau_j))}\,\ds\nonumber\\
   &=\sum_{\ell=1}^{N_{q}}\frac{\tau_j^2}{\tau_{j+1}(\tau_{j}+\tau_{j+1})}\int_0^1\varpi^{\ell}e^{-\theta^{\ell}(t_k^*-(t_{j-1}+s \tau_j))} {\rm d}(s^2-s)\\
  &=\sum_{\ell=1}^{N_{q}} \frac{\varpi^{\ell}\theta^{\ell}\tau_{j}^3}{\tau_{j+1}(\tau_{j}+\tau_{j+1})}\int_0^{1} s(1-s)e^{-\theta^{\ell}(t_k^*-t_{j}+s \tau_j)}\  \ds.\nonumber
  \end{align}
From equations \eqref{eq:akj1_rem}--\eqref{eq:ckj_rem}, it is easy to verify $\hat{a}^{(k)}_{j}<0$ and  $\hat{c}^{(k)}_{j}>0$, which imply $\hat{d}^{(k)}_{j}=\hat{c}^{(k)}_{j-1}-\hat{a}^{(k)}_{j}>0$, i.e., (P1) holds (note that $\hat{c}^{(k)}_{0}=0$). Moreover, for any fixed $s$, we know that
 $ e^{-\theta^{\ell}(t_k^*-t_{j-1})}$, $ e^{-\theta^{\ell}(t_k^*-t_{j-1}-s \tau_j)}$, and $ e^{-\theta^{\ell}(t_k^*-t_{j}+s \tau_j)}$
decrease w.r.t. $k$. Combining this with  $\theta^{\ell}>0$ and  $\varpi^{\ell}>0$ stated in Lemma~\ref{lem:SOE_approx}, we can claim that 
$\hat{a}^{(k+1)}_{j}-\hat{a}^{(k)}_{j}>0$, $\hat{c}^{(k+1)}_{j}-\hat{c}^{(k)}_{j}<0$ and $\hat{d}^{(k+1)}_{j}-\hat{d}^{(k)}_{j}<0$, i.e., (P2) holds.

We now turn to prove the properties (P3) and (P4). From equations \eqref{eq:akj_rem}--\eqref{eq:ckj_rem}, we can derive
\begin{equation}\label{ineq:dkjdiff}
\begin{aligned}
    &\hat{d}^{(k)}_{j+1}-\hat{d}^{(k)}_{j}= (\hat{c}^{(k)}_{j}-\hat{c}^{(k)}_{j-1})-(\hat{a}_{j+1}^{(k)}-\hat{a}^{(k)}_{j})\\
    =&\sum_{\ell=1}^{N_{q}} \varpi^{\ell}\theta^{\ell}\bigg(G_{j,k}^{\ell}
    +\frac{\tau_{j+1}}{\tau_{j+1}+\tau_{j+2}}
     \int_0^{1}(\tau_{j+1}+\tau_{j+2}-s\tau_{j+1}) (1-s)e^{-\theta^{\ell}(t_k^*-t_{j}-s \tau_{j+1})}\,\ds\bigg),   \end{aligned}
\end{equation}
where 
\begin{equation}\label{eq:gjk}
\begin{aligned}
     G_{j,k}^{\ell}=&\frac{\tau_{j}^3}{\tau_{j+1}(\tau_{j}+\tau_{j+1})}\int_0^{1} s(1-s)e^{-\theta^{\ell}(t_k^*-t_{j}+s \tau_j)}\  \ds\\
    & -\frac{\tau_{j-1}^3}{\tau_{j}(\tau_{j-1}+\tau_{j})}\int_0^{1} s(1-s)e^{-\theta^{\ell}(t_k^*-t_{j-1}+s \tau_{j-1})}\  \ds\\
    &+ \frac{\tau_j}{\tau_j+\tau_{j+1}}
    \int_0^{1} (\tau_j+\tau_{j+1}+s\tau_j)
    (1-s)e^{-\theta^{\ell}(t_k^*-t_{j}+s \tau_j)}\,\ds\\
    =&\left(\frac{\tau_{j}^3}{\tau_{j+1}(\tau_{j}+\tau_{j+1})}+\frac{(4\tau_j+3\tau_{j+1})\tau_j}{\tau_j+\tau_{j+1}}\right)\int_0^{1} s(1-s)e^{-\theta^{\ell}(t_k^*-t_{j}+s \tau_j)}\  \ds\\
    &- \frac{\tau_{j-1}^3}{\tau_{j}(\tau_{j-1}+\tau_{j})}\int_0^{1} s(1-s)e^{-\theta^{\ell}(t_k^*-t_{j-1}+s \tau_{j-1})}\  \ds\\
    &+ e^{-\theta^{\ell}(t_k^*-t_{j})}\tau_j
     \int_0^{1} (1-3s)(1-s)e^{-\theta^{\ell}s \tau_j} \,\ds.
\end{aligned}
\end{equation}
and  we use the forms \eqref{eq:akj_rem} for $\hat{a}^{(k)}_{j}$ and \eqref{eq:akj1_rem} for $\hat{a}^{(k)}_{j+1}$. Here we make a convention $\tau_0=0$, so that $\hat{d}^{(k)}_{1}=-\hat{a}^{(k)}_{1}.$
Under the condition \eqref{condition:rho}, we have 
\begin{equation*}
\begin{aligned}
  \frac{(4\tau_j+3\tau_{j+1})\tau_j}{\tau_j+\tau_{j+1}}-\frac{\tau_{j-1}^3}{\tau_{j}(\tau_{j-1}+\tau_{j})}
  = \tau_j\left(\frac{1}{1+\rho_{j+1}}+\frac{3\rho_{j}^2(1+\rho_{j})-1}{\rho_{j}^2(1+\rho_{j})}\right)\ge 0.
\end{aligned}
\end{equation*} Combining this with the fact 
$
    e^{-\theta^{\ell}(t_k^*-t_{j}+s \tau_j)} >e^{-\theta^{\ell}(t_k^*-t_{j-1}+s \tau_{j-1})}
$
 and 
\begin{align*}
    &\theta^{\ell} \tau_{j}\int_0^{1} (1-3s)
    (1-s)e^{-\theta^{\ell}s \tau_{j}}\,\ds
    =\int_0^{1} (1-3s)
    (1-s)\left(-e^{-\theta^{\ell}s \tau_{j}}\right)'\,\ds\\
    =&1+\int_0^{1}(6s-4)e^{-\theta^{\ell}s \tau_{j}}\,\ds
        =1-e^{-\theta^{\ell}\xi_j^{\ell}\tau_j},
\end{align*}
where  $\xi_j^{\ell}$ is a constant in $(0,1)$,
we can find that 
\begin{equation}\label{eq:gjkineq}
    G_{j,k}^{\ell}>(\theta^{\ell})^{-1}e^{-\theta^{\ell}(t_k^*-t_j)}\left(1-e^{-\theta^{\ell}\xi_j^{\ell}\tau_j}\right)>0,\quad \forall j< k.
\end{equation}
Then we can derive from \eqref{ineq:dkjdiff} that
\begin{equation}\label{eq:P9}
     \hat{d}^{(k)}_{j+1}-\hat{d}^{(k)}_{j}
        >\sum_{\ell=1}^{N_{q}} \frac{\varpi^{\ell}\theta^{\ell}\tau_{j+1}}{\tau_{j+1}+\tau_{j+2}}
     \int_0^{1}(\tau_{j+1}+\tau_{j+2}-s\tau_{j+1}) (1-s)e^{-\theta^{\ell}(t_k^*-t_{j}-s \tau_{j+1})}\,\ds.
\end{equation}
Therefore, (P3) holds. Moreover,   the property (P4) holds follows from
{\small
 \begin{align*}
    & (\hat{d}^{(k)}_{j+1}-\hat{d}^{(k)}_{j})-(\hat{d}^{(k+1)}_{j+1}-\hat{d}^{(k+1)}_{j} )\\
    =&\sum_{\ell=1}^{N_{q}} \varpi^{\ell}\theta^{\ell}\left(1-e^{\theta^{\ell}(t_{k}^*-t_{k+1}^*)}\right)\\ &\qquad\bigg(G_{j,k}^{\ell}
   +\frac{\tau_{j+1}}{\tau_{j+1}+\tau_{j+2}}
     \int_0^{1}(\tau_{j+1}+\tau_{j+2}-s\tau_{j+1}) (1-s)e^{-\theta^{\ell}(t_k^*-t_{j}-s \tau_{j+1})}\,\ds\bigg)>0, 
 \end{align*}}
 according to  \eqref{ineq:dkjdiff}. 
\end{proof}
Now we can state and prove the following main theorem.
\begin{theorem}\label{thm1}
Consider a nonuniform mesh $\{\tau_k\}_{k\ge 1}$ satisfying that for $k\ge 2$
\begin{equation}\label{thm1cond}
\begin{aligned}
    \rho_k&\ge \eta \approx0.475329,
    &&\varepsilon\le \min _{k\ge 1}\frac{1}{5(1-\alpha)(\sigma\tau_k)^\alpha},\\
    \Delta t& \le \min_{k\ge 2}\sigma \tau_k, &&T_{\rm{soe}}\ge \max_{k\ge2 }(\sigma\tau_{k+1}+\tau_k),
\end{aligned}
\end{equation}
 where $\varepsilon$,  $\Delta t$, and $T_{\rm{soe}}$ are fixed in Lemma~\ref{lem:SOE_approx}, and $\sigma=1-\alpha/2$.
Then for any function $u$ defined on $[0,\infty)\times\Omega$ and $n\ge 1$,
\begin{equation*}\label{eq:PD}
    \mathcal B_n(u,u) = \sum_{k=1}^{n}\langle F_k^{\alpha,*} u, \delta_k u\rangle\geq \sum_{k=1}^{n} \frac{[\mathbf B]_{kk}}{\Gamma(1-\alpha)} \|\delta_k u\|^2\geq 0,
\end{equation*}
where  $\|\cdot\|$ is the $L^2$-norm and $\mathbf B$ is defined in \eqref{matrxB}.

\end{theorem}
\begin{proof}
According to \eqref{eq:soeFk}, we can rewrite $\mathcal B_n(u,u)$ in the following matrix form
\begin{equation}\label{matM}
    \mathcal B_n(u,u) = \sum_{k=1}^{n}\langle F_k^{\alpha,*} u, \delta_k u\rangle=\frac{1}{\Gamma(1-\alpha)}\int_{\Omega}\psi \mathbf M \psi^{\mathrm{T}}\dx,
\end{equation}
where 
$
    \psi=[\delta_{1} u,\delta_2 u,\cdots,\delta_n u]
$
and
$\mathbf M = \mathbf A+\mathbf B$ with
{\footnotesize
\begin{equation*}
    \mathbf A=
\begin{pmatrix}
\beta_1 & \\
\hat{d}^{(2)}_{1}& \beta_2  \\
\hat{d}^{(3)}_{1} & \hat{d}^{(3)}_{2}& \beta_3 \\
\vdots& \vdots & \ddots & \ddots\\
\hat{d}^{(n)}_{1}& \hat{d}^{(n)}_{2}&\cdots&  \hat{d}^{(n)}_{n-1}& \beta_n
\end{pmatrix}
\end{equation*}}
and
\begin{equation}\label{matrxB}
    \mathbf B={\rm diag}\left(\frac{\sigma^{1-\alpha}}{(1-\alpha)\tau_1^{\alpha}}-\beta_1,~\frac{\sigma^{1-\alpha}}{(1-\alpha)\tau_2^{\alpha}}+\hat{c}^{(2)}_{1}-\beta_2,~\cdots,~\frac{\sigma^{1-\alpha}}{(1-\alpha)\tau_n^{\alpha}}+\hat{c}^{(n)}_{n-1}-\beta_n\right)
\end{equation}
with
\begin{equation}\label{betak}
\begin{aligned}
    2\beta_1= \hat{d}^{(2)}_{1},\quad
    2\beta_k -\hat{d}^{(k+1)}_{k}=\hat{d}^{(k)}_{k-1}-\hat{d}^{(k+1)}_{k-1},\quad 2\leq k\leq n.
\end{aligned}
\end{equation}
Consider the following symmetric matrix 
$
    \mathbf S = \mathbf A+\mathbf A^{\mathrm T}.
$
According to  the properties (P1)--(P4) in Lemma \ref{lemmapro}, if the condition \eqref{thm1cond} holds, all the elements of $ \mathbf S$ are positive, and  $\mathbf S$ satisfies the following three properties:
\begin{itemize}
\item[{\rm (1)}] $\forall\; 1\leq j < i \leq n$, $\left[ \mathbf S \right]_{i-1,j}\geq \left[ \mathbf S \right]_{i, j}$;
\item[{\rm (2)}] $\forall\; 1 < j \leq i \leq n$, $\left[ \mathbf S \right]_{i, j-1}< \left[ \mathbf S \right]_{i, j}$;
\item[{\rm (3)}] $\forall \;1< j < i \leq n$, $\left[ \mathbf S \right]_{i-1, j-1} - \left[ \mathbf S \right]_{i, j-1}\leq \left[ \mathbf S \right]_{i-1, j} - \left[ \mathbf S \right]_{i, j}$.
\end{itemize}
From \cite[Lemma 2.1]{CSIAM-AM-1-478}, $\mathbf S=\mathbf A+\mathbf A^{\mathrm T}$ is positive semidefinite.

Next we will prove $[\mathbf B]_{kk}\ge0$, $k\ge 1$, under the condition \eqref{thm1cond}. When $k=1$, from the second equality in \eqref{eq:akj1_rem}, $2\beta_1=\hat{d}^{(2)}_{1}= -\hat{a}^{(2)}_1$ in \eqref{betak}, and Lemma~\ref{lem:SOE_approx}, we have
\begin{align}\label{ineq:B11}
    [\mathbf B]_{11}&=\frac{\sigma^{1-\alpha}}{(1-\alpha)\tau_1^{\alpha}}-\frac12\sum_{\ell=1}^{N_{q}} \int_0^1 \frac{2 s+\rho_2}{1+\rho_{2}}\varpi^{\ell}e^{-\theta^{\ell}(\sigma\tau_2+s \tau_1)}\,\ds\\
    &> \frac{\sigma^{1-\alpha}}{(1-\alpha)\tau_1^{\alpha}}-\frac12\sum_{\ell=1}^{N_{q}}\varpi^{\ell}e^{-\theta^{\ell}\sigma\tau_2}\int_0^1\frac{2s+\rho_2}{1+\rho_2}\ds
   =\frac{\sigma^{1-\alpha}}{(1-\alpha)\tau_1^{\alpha}}-\frac12\sum_{\ell=1}^{N_{q}}\varpi^{\ell}e^{-\theta^{\ell}\sigma\tau_2}\nonumber\\&\ge \frac{\sigma^{1-\alpha}}{(1-\alpha)\tau_1^{\alpha}}-\frac12(\sigma\tau_2)^{-\alpha}-\frac12\varepsilon\ge\frac{1}{2(1-\alpha)(\sigma\tau_1)^\alpha}\left( 2-\alpha-(1-\alpha)\eta^{-\alpha}\right)-\frac12\varepsilon.\nonumber
    \end{align}
  When $k\ge 2$,  from $2\beta_k -\hat{d}^{(k+1)}_{k}=\hat{d}^{(k)}_{k-1}-\hat{d}^{(k+1)}_{k-1}$ in \eqref{betak},  we have  
\begin{align} \label{eq:bkk}
     [\mathbf B]_{kk}
    =& \frac{\sigma^{1-\alpha}}{(1-\alpha)\tau_k^{\alpha}}+\frac12\hat{c}^{(k)}_{k-1}+\frac12\left[\hat{c}^{(k)}_{k-1}-(\hat{d}^{(k)}_{k-1}-\hat{d}^{(k+1)}_{k-1}+\hat{d}^{(k+1)}_{k})\right].
\end{align}
From the equations \eqref{eq:akj_rem}--\eqref{eq:ckj_rem} and $\hat{d}^{(k)}_{j}=\hat{c}^{(k)}_{j-1}-\hat{a}^{(k)}_{j}$, when the condition \eqref{thm1cond} holds, we can derive for $k\ge 2$
\begin{equation}
 \begin{aligned}
 &\hat{c}^{(k)}_{k-1}-(\hat{d}^{(k)}_{k-1}-\hat{d}^{(k+1)}_{k-1}+\hat{d}^{(k+1)}_{k})\\
    =&(\hat{c}^{(k)}_{k-1}-\hat{c}^{(k+1)}_{k-1})-(\hat{c}^{(k)}_{k-2}-\hat{c}^{(k+1)}_{k-2} )+(\hat{a}^{(k)}_{k-1}-\hat{a}^{(k+1)}_{k-1}+\hat{a}^{(k+1)}_{k})\\
    =&\sum_{\ell=1}^{N_{q}} \varpi^{\ell}\theta^{\ell}\left(1-e^{\theta^{\ell}(t_{k}^*-t_{k+1}^*)}\right)G_{k-1,k}^{\ell}
    -\sum_{\ell=1}^{N_{q}}  \varpi^{\ell}e^{-\theta^{\ell}(\sigma\tau_k)}\\&-\sum_{\ell=1}^{N_{q}} \frac{\varpi^{\ell}\theta^{\ell}\tau_{k}}{\tau_{k}+\tau_{k+1}}
     \int_0^{1}(\tau_k+\tau_{k+1}-s\tau_{k}) (1-s)e^{-\theta^{\ell}(t_{k+1}^*-t_{k-1}-s \tau_{k})}\,\ds \\
   \ge &-\sum_{\ell=1}^{N_{q}}  \varpi^{\ell}e^{-\theta^{\ell}(\sigma\tau_k)}-\sum_{\ell=1}^{N_{q}} \frac{\varpi^{\ell}\theta^{\ell}\tau_k}{\tau_k+\tau_{k+1}}
     \int_0^{1}(\tau_{k+1}+s\tau_{k}) s e^{-\theta^{\ell}(\sigma\tau_{k+1}+s\tau_{k})}\,\ds,
    \end{aligned}
\end{equation}
    where $G_{k-1,k}^{\ell}>0$ is defined in \eqref{eq:gjk} and we use  \eqref{eq:akj_rem} for $\hat{a}^{(k)}_{k-1}$ and  $\hat{a}^{(k+1)}_{k-1}$, and \eqref{eq:akj1_rem} for $\hat{a}^{(k+1)}_{k}$.
     Moreover, from Lemma~\ref{lem:SOE_approx}, we have
     \begin{align}\label{ineq:hbk}
 &\sum_{\ell=1}^{N_{q}}  \varpi^{\ell}e^{-\theta^{\ell}\sigma\tau_k}+ \sum_{\ell=1}^{N_{q}} \frac{\varpi^{\ell}\theta^{\ell}\tau_k}{\tau_k+\tau_{k+1}}
     \int_0^{1}(\tau_{k+1}+s\tau_{k}) s e^{-\theta^{\ell}(\sigma\tau_{k+1}+s\tau_{k})}\,\ds\\
     \le&\sum_{\ell=1}^{N_{q}}  \varpi^{\ell}\left(e^{-\theta^{\ell}\sigma\tau_k}+ \theta^{\ell}\tau_k
     \int_0^{1}  e^{-\theta^{\ell}(\sigma\tau_{k+1}+s\tau_{k})}\,\ds\right)\nonumber\\=&\sum_{\ell=1}^{N_{q}}  \varpi^{\ell}\left(e^{-\theta^{\ell}\sigma\tau_k}+ \int_0^{1}  (-e^{-\theta^{\ell}(\sigma\tau_{k+1}+s\tau_{k})})'\,\ds\right)\nonumber\\
     =&\sum_{\ell=1}^{N_{q}}  \varpi^{\ell}\left(e^{-\theta^{\ell}\sigma\tau_k}+ e^{-\theta^{\ell}\sigma\tau_{k+1}}-e^{-\theta^{\ell}(\sigma\tau_{k+1}+\tau_k)}\right)\nonumber\\
     \le&(\sigma\tau_k)^{-\alpha}+(\sigma\tau_{k+1})^{-\alpha}-(\sigma\tau_{k+1}+\tau_k)^{-\alpha}+3\varepsilon \nonumber
\end{align}
    Combining \eqref{eq:bkk}--\eqref{ineq:hbk}, $\sigma=1-\alpha/2$  and the fact $\hat{c}^{(k)}_{k-1}>0$, we can derive  
     \begin{equation}\label{ineq:Bkk}
     \begin{aligned}
      [\mathbf B]_{kk} 
     \ge& \frac{\sigma^{1-\alpha}}{(1-\alpha)\tau_k^{\alpha}}-\frac12(\sigma\tau_k)^{-\alpha}-\frac12(\sigma\tau_{k+1})^{-\alpha}+\frac12(\sigma\tau_{k+1}+\tau_k)^{-\alpha}-\frac32\varepsilon\\
    =&\frac{1}{2(1-\alpha)(\sigma\tau_k)^\alpha}\left(1-(1-\alpha)\sigma^{\alpha}[(\sigma\rho_{k+1})^{-\alpha}-(\sigma\rho_{k+1}+1)^{-\alpha}]\right)-\frac32\varepsilon\\
    \ge&\frac{1}{2(1-\alpha)(\sigma\tau_k)^\alpha}\left(1-(1-\alpha)\sigma^{\alpha}[(\sigma\eta)^{-\alpha}-(\sigma\eta+1)^{-\alpha}]\right)-\frac32\varepsilon.
     \end{aligned}
     \end{equation}
We now plot the following functions in Figure \ref{fig:valuef12}
\begin{equation*}
   f_1(\alpha)= 2-\alpha-(1-\alpha)\eta^{-\alpha},\quad f_2(\alpha)=1-(1-\alpha)\sigma^{\alpha}[(\sigma\eta)^{-\alpha}-(\sigma\eta+1)^{-\alpha}],
\end{equation*}
where $\eta \approx 0.475329$ from \eqref{condition:rho}.
We can see that $f_i(\alpha)\ge0.6$, for $i=1,\ 2$.  Thus  $[\mathbf B]_{kk}\ge 0$, $k\ge 1$, under the condition \eqref{thm1cond}. 
     \begin{figure}
    \centering
    \includegraphics[width=0.45\textwidth]{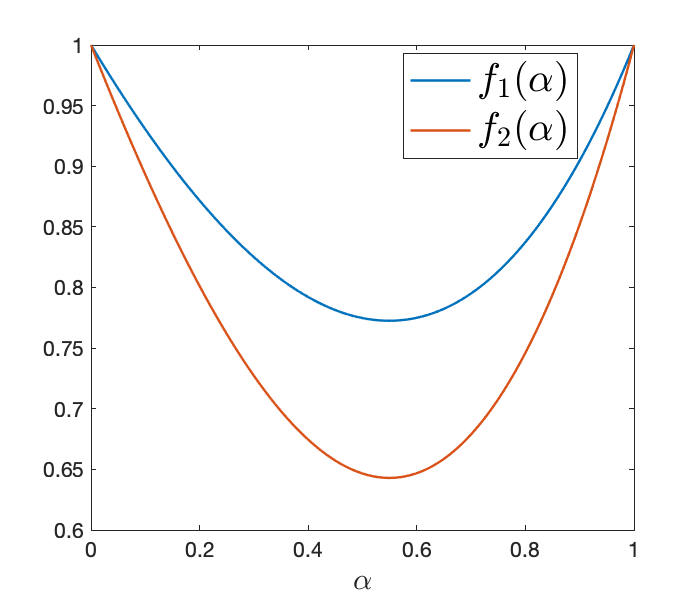}
   \vspace{-0.2in}    
    \caption{$f_1(\alpha)$ and $f_2(\alpha)$ w.r.t. $\alpha\in (0,1)$.}
    \label{fig:valuef12}
\end{figure}

\end{proof}

\section{Long time $H^1$-stability of the fast  L2-1$_\sigma$ scheme}\label{sect4}
\subsection{Linear subdiffusion equation}
For simplicity, we first consider the following linear subdiffusion equation:
\begin{equation}\label{eq:subdiffusion}
\begin{aligned}
    \partial_t^\alpha u(t,x) & = \Delta u(t,x)+f(t,x),&& (t,x)\in (0,\infty)\times\Omega,\\
    u(t,x) &= 0,&& (t,x)\in (0,\infty)\times\partial\Omega,\\
    u(0,x)& = u^0(x),&& x\in \Omega,
\end{aligned}    
\end{equation}
where $\Omega$ is a bounded Lipschitz domain in $\mathbb R^d$. 
Given an arbitrary nonuniform mesh $\{\tau_k\}_{k\geq 1}$, the fast L2-$1_\sigma$ scheme of this subdiffusion equation is written as
\begin{equation}\label{eq:sch_sub}
\begin{aligned}
    F_k^{\alpha,*} u 
    &= (1-\alpha/2)\Delta u^k+\alpha/2\Delta u^{k-1}+f(t_k^*,\cdot),&&  \text{in}~ \Omega,\\
    u^k&=0,&&\text{on} ~\partial\Omega.
\end{aligned}    
\end{equation}
\begin{theorem}\label{thm2}
Assume that $f(t,x) \in L^\infty([0,\infty);L^2(\Omega)) \cap BV([0,\infty); L^2(\Omega))$ is a bounded variation function in time and $ u^0\in H_0^1(\Omega)$.
Fix $\varepsilon$,  $\Delta t$, and $T_{\rm{soe}}$ in Lemma~\ref{lem:SOE_approx}.
For an arbitrary nonuniform mesh $\{\tau_k\}_{k\geq 1}$ satisfying \eqref{thm1cond}, then the numerical solution $u^n$ of the fast  L2-$1_\sigma$ scheme~\eqref{eq:sch_sub} satisfies the following global-in-time  $H^1$-stability
\begin{equation*}
\begin{aligned}
    \|\nabla u^n\|
   &\le\|\nabla u^0\| +2C_fC_{\Omega}, \quad \forall n\geq 1,
\end{aligned}
\end{equation*}
where  $C_{f}= 2\|f\|_{L^\infty([0,\infty);L^2(\Omega))}+\|f\|_{BV([0,\infty); L^2(\Omega))}$, and $C_\Omega$ is the Sobolev embedding constant depending on $\Omega$ and the spatial dimension $d$.
\end{theorem}
\begin{proof}
Due to the positive semidefiniteness of bilinear form $\mathcal B_n$ established in Theorem \ref{thm1}, the proof is quite similar as the one of \cite[Theorem 4.1]{quan2022energy}. We omit it here.
\end{proof}

\subsection{Semilinear subdiffusion equation}
In this part, we consider the semilinear subdiffusion  equation
\begin{equation} \label{eq:nonlinear}
    \begin{aligned}
        \partial_t^\alpha u(t,x)& = \nu^2\Delta u(t,x) -f(u(t,x)),\quad &&(t,x)\in (0,\infty)\times \Omega,\\
         u(0,x)& = u^0(x),&& x\in \Omega,
    \end{aligned}
\end{equation}
where $\nu >0$ is the  diffusion coefficient, $f$ is some nonlinear functional of $u(t,x)$, and $\Omega$ is a bounded Lipschitz domain in $\mathbb R^d$. 
For simplicity, we consider the periodic boundary condition or the homogeneous Dirichlet/Neumann boundary condition and assume that 
\begin{equation}\label{eq:lipschitz}
    \sup_{u\in\mathbb R} |f'(u)|\leq L
\end{equation}
 for some constant $L> 0$. One of such models is the time fractional phase-field model, which is revealed to admit different dynamical scales \cite{tangwangyang}.
 
We consider the following scheme for the  semilinear subdiffusion equation \eqref{eq:nonlinear} in $\Omega$, using the fast L2-$1_\sigma$ formula for Caputo derivative and the Newton linearization for nonlinearity:
\begin{equation}\label{eq:sch_subnonlinear}
    F_k^{\alpha,*} u 
    = (1-\alpha/2) \nu^2\Delta u^k+(\alpha/2)\nu^2\Delta u^{k-1}-f(u^{k-1})-(1-\alpha/2) f'(u^{k-1})\delta_ku.
\end{equation}
\begin{theorem}\label{thm:nonlinear}
Assume that $f$ in \eqref{eq:nonlinear} satisfies \eqref{eq:lipschitz}. 
Fix $\varepsilon$,  $\Delta t$, and $T_{\rm{soe}}$ in Lemma~\ref{lem:SOE_approx}.
For an arbitrary nonuniform mesh $\{\tau_k\}_{k\geq 1}$ satisfying \eqref{thm1cond}
and
\begin{equation}\label{cond:tau}
\begin{aligned}
&\tau_1^\alpha \leq \frac{2-\alpha-(1-\alpha)\eta^{-\alpha}}{\sigma^\alpha [(3-\alpha)\Gamma(2-\alpha)L+(1-\alpha)\varepsilon]}, && \\
 &   \tau_k^\alpha\leq \frac{1-(1-\alpha)\sigma^{\alpha}[(\sigma\eta)^{-\alpha}-(\sigma\eta+1)^{-\alpha}]}{\sigma^\alpha [(3-\alpha)\Gamma(2-\alpha)L+3(1-\alpha)\varepsilon]},&& k\geq 2,
    \end{aligned}
\end{equation}
where $\sigma=1-\alpha/2$, then the numerical solution $u^n$ of the fast L2-1$_\sigma$ scheme  \eqref{eq:sch_subnonlinear} satisfies the following energy stability: $\forall n\ge 1$,
\begin{equation*}
\begin{aligned}
  \int_{\Omega}\left(\frac{\nu^2}{2}|\nabla u^n|^2+F(u^n)\right)\ \dx\le \int_{\Omega}\left(\frac{\nu^2}{2}|\nabla u^0|^2+F(u^0)\right)\ \dx,
\end{aligned}
\end{equation*}
where $F$ is a primitive integral of $f$, i.e., $F'(u) = f(u)$. 
\end{theorem}
\begin{proof}
Fix $n$ and let  
$
    E^n\coloneqq\int_{\Omega}\left(\frac{\nu^2}{2}|\nabla u^n|^2+F(u^n)\right)\dx.
$ 
Multiplying \eqref{eq:sch_subnonlinear} by $\delta_k u$ and integrating over $\Omega$ yield
\begin{equation*}
\begin{aligned}
&\langle F_k^{\alpha,*} u, \delta_k u \rangle\\=&
    -\frac{\nu^2}{2}(\|\nabla u^k\|^2-\|\nabla u^{k-1}\|^2+(1-\alpha)\|\nabla \delta_k u\|^2)\\& -\langle F(u^k)- F(u^{k-1}),1\rangle +\frac12\langle f'(\omega)\delta_k u,\delta_k u\rangle
 -(1-\alpha/2) \langle f'(u^{k-1})\delta_k u,\delta_k u\rangle\\
   =&-E^k+E^{k-1}-\frac{\nu^2(1-\alpha)}{2}\|\nabla \delta_k u\|^2+\frac12\langle (f'(\omega)-2(1-\alpha/2) f'(u^{k-1}))\delta_k u,\delta_k u\rangle\\
   \le& -E^k+E^{k-1}+\frac{(3-\alpha)L}{2}\|\delta_k u\|^2,
   \end{aligned}    
\end{equation*}
where $\omega$ is between $u^k$ and $u^{k-1}$. 
Summing up the above inequality over $k$ yields
\begin{equation*}\label{eq:sch_implicitenergy}
    \sum_{k=1}^n\langle F_k^\alpha u, \delta_k u \rangle \le -E^n+E^0+ \frac{(3-\alpha)L}{2}\sum_{k=1}^n\|\delta_k u\|^2.
\end{equation*}
From Theorem \ref{thm1}, we have 
\begin{equation*}
    E^n\le E^0 -\sum_{k=1}^n
    \left(\frac{[\mathbf B]_{kk}}{\Gamma(1-\alpha)}-\frac{(3-\alpha)L}{2}\right)\|\delta_k u\|^2.
\end{equation*}
According to \eqref{ineq:B11} and \eqref{ineq:Bkk}, under the condition
\eqref{cond:tau},
we have 
\begin{equation*}
    \frac{[\mathbf B]_{kk}}{\Gamma(1-\alpha)}-\frac{(3-\alpha)L}{2}\geq 0,\quad 1\leq k\leq n
\end{equation*}
and consequently $E^n\le E^0$ for any $n\ge 1$.
\end{proof}

\section{Sharp finite time $H^1$-error estimate} \label{sect5}
In this part, we present the sharp $H^1$-error estimate of the fast L2-1$_\sigma$ scheme \eqref{eq:sch_sub} for the linear subdiffusion equation \eqref{eq:subdiffusion} in finite time  $t\in(0,T_{\rm{soe}}]$ with some constraints on the time step. We consider time meshes  $\{\tau_k\}_{k=1}^N$:
\begin{equation}\label{mesh}
   0=t_0<t_1<\ldots<t_k\ldots<t_N=T_{\rm{soe}},\quad \text{with}\quad\tau_k = t_k-t_{k-1},
\end{equation}
where $T_{\rm{soe}}$ is given in Lemma~\ref{lem:SOE_approx}. 
We first reformulate the discrete formula \eqref{eq:soeFk}:
\begin{equation*}\label{eq:equa}
    F^{\alpha,*}_k u=\frac{1}{\Gamma(1-\alpha)}\left( [\mathbf M]_{k,k} u^k - \sum_{j=2}^k([\mathbf M]_{k,j}-[\mathbf M]_{k,j-1}) u^{j-1}-[\mathbf M]_{k,1}u^0\right),
\end{equation*}
where $\mathbf M$ is given by \eqref{matM}.
We now give some properties on $[\mathbf M]_{k,j}$.
\begin{lemma}\label{proMkj}
Given a nonuniform mesh  $\{\tau_k\}_{k=1}^N$ in \eqref{mesh} satisfying 
\begin{equation}\label{thm2cond}
\begin{aligned}
    \rho_k\ge \eta \approx0.475329,\quad\Delta t\le \min_{2\le k\le N}\sigma \tau_k,\quad
    \varepsilon\le\min_{2\le k\le N}\frac{1-\alpha}{\sigma}\hat{c}^{(k)}_{k-1},
\end{aligned}
\end{equation}
where   $\varepsilon,\ \Delta t$ are fixed in Lemma~\ref{lem:SOE_approx} and $\sigma=1-\alpha/2$, then following properties of $[\mathbf M]_{k,j}$ given by \eqref{matM} hold:
\begin{itemize}
\item[{\rm (Q1)}] for $1\le j\le k-1,$
{\small
\begin{equation*}
    [\mathbf M]_{k,j}\ge \frac{\eta}{(1+\eta)\tau_j}\sum_{\ell=1}^{N_{q}} \int_{t_{j-1}}^{t_j} \varpi^{\ell}e^{-\theta^{\ell}(t_k^*-s)}\,\ds,\quad  [\mathbf M]_{k,k}\ge\frac{1}{\tau_k}\int_{t_{k-1}}^{t_k^*}(t_k^*-s)^{-\alpha}\ds;
\end{equation*}}
\item[{\rm (Q2)}] 
for $2\le j\le k-1$, 
{\small
\begin{align*}
 [\mathbf M]_{k,j}-[\mathbf M]_{k,j-1}&\ge    \sum_{\ell=1}^{N_{q}} \frac{\varpi^{\ell}\theta^{\ell}\tau_{j}}{\tau_{j}+\tau_{j+1}}
     \int_0^{1}(\tau_{j}+\tau_{j+1}-s\tau_{j}) (1-s)e^{-\theta^{\ell}(t_k^*-t_{j-1}-s \tau_{j})}\,\ds,\\
      [\mathbf M]_{k,k}-[\mathbf M]_{k,k-1}&\ge\frac{\alpha}{2(1-\alpha)(\sigma\tau_{k})^{\alpha}}-\varepsilon;
\end{align*}}
\item[{\rm (Q3)}] 
{\small
\begin{equation*}
       \frac{1-\alpha}{\sigma}[\mathbf M]_{k,k}-[\mathbf M]_{k,k-1}\ge 0, \quad k\ge 2. 
\end{equation*}}
\end{itemize}
\end{lemma}
\begin{proof}  
From \eqref{matM} and the first equation in \eqref{eq:akj_rem}, for $1\le j\le k-1$, 
\begin{equation*}\label{ineq:mrho2}
     \begin{aligned}
      [\mathbf M]_{k,j}&=\hat{d}^{(k)}_{j}\ge- \hat{a}^{(k)}_{j}
      \ge\sum_{\ell=1}^{N_{q}} \int_0^1 \frac{\tau_{j+1}}{\tau_{j}+\tau_{j+1}}\varpi^{\ell}e^{-\theta^{\ell}(t_k^*-(t_{j-1}+s \tau_j))}\,\ds\\& \ge \frac{\eta}{(1+\eta)\tau_j}\sum_{\ell=1}^{N_{q}} \int_{t_{j-1}}^{t_j} \varpi^{\ell}e^{-\theta^{\ell}(t_k^*-s)}\,\ds,
     \end{aligned}
\end{equation*}
     and  
     \begin{equation*}
          [\mathbf M]_{k,k}=\hat{c}^{(k)}_{k-1}+\frac{\sigma^{1-\alpha}}{(1-\alpha)\tau_k^\alpha}\ge\frac{\sigma^{1-\alpha}}{(1-\alpha)\tau_k^\alpha}=\frac{1}{\tau_k}\int_{t_{k-1}}^{t_k^*}(t_k^*-s)^{-\alpha}\ds.
     \end{equation*}

For $ 2\le j\le k-1$,
$[\mathbf M]_{k,j}-[\mathbf M]_{k,j-1}=\hat{d}^{(k)}_{j}-\hat{d}_{j-1}^{(k)}$. Then the first the result of (Q2) is from  \eqref{eq:P9}.
 Moreover from Lemma~\ref{lem:SOE_approx}, if \eqref{condition:rho} holds for $j=k-1$, we have
    \begin{align*}
  & [\mathbf M]_{k,k}-[\mathbf M]_{k,k-1}
    =\frac{\sigma^{1-\alpha}}{(1-\alpha)\tau_k^{\alpha}}+(\hat{c}^{(k)}_{k-1}-\hat{c}^{(k)}_{k-2}+\hat{a}^{(k)}_{k-1})\\
    =&\frac{\sigma^{1-\alpha}}{(1-\alpha)\tau_k^{\alpha}}+\sum_{\ell=1}^{N_{q}} \varpi^{\ell}\theta^{\ell}G_{k-1,k}^{\ell}-\sum_{\ell=1}^{N_{q}} \varpi^{\ell}e^{-\theta^{\ell}\sigma\tau_k}\\
     >&\frac{\sigma^{1-\alpha}}{(1-\alpha)\tau_k^{\alpha}}-\sum_{\ell=1}^{N_{q}} \varpi^{\ell}e^{-\theta^{\ell}\sigma\tau_k}\ge \frac{\sigma^{1-\alpha}}{(1-\alpha)\tau_k^{\alpha}}-(\sigma\tau_k)^{-\alpha}-\varepsilon=\frac{\alpha}{2(1-\alpha)(\sigma\tau_{k})^{\alpha}}-\varepsilon.
    \end{align*}
  where $G_{k-1,k}^{\ell}>0$ is defined in \eqref{eq:gjk}.
For the property (Q3),  from \eqref{eq:akj_rem} and \eqref{eq:ckj_rem} and Lemma~\ref{lem:SOE_approx}  we have {\small
 \begin{align*}
  &  \frac{1-\alpha}{\sigma}[\mathbf M]_{k,k}-[\mathbf M]_{k,k-1}
   \\
   =&(\sigma\tau_k)^{-\alpha}+\frac{1-\alpha}{\sigma}\hat{c}^{(k)}_{k-1}-
    \sum_{\ell=1}^{N_{q}} \varpi^{\ell}\theta^{\ell}\bigg(  \frac{\tau_{k-2}^3}{\tau_{k-1}(\tau_{k-2}+\tau_{k-1})}\int_0^{1} s(1-s)e^{-\theta^{\ell}(t_k^*-t_{k-2}+s \tau_{k-2})}\  \ds\\
    &- \frac{\tau_{k-1}}{\tau_{k-1}+\tau_{k}}
    \int_0^{1} (\tau_{k-1}+\tau_{k}+s\tau_{k-1})
    (1-s)e^{-\theta^{\ell}(t_k^*-t_{k-1}+s \tau_{k-1})}\,\ds\bigg)-\sum_{\ell=1}^{N_{q}} \varpi^{\ell}e^{-\theta^{\ell}\sigma\tau_k}\\
     \ge& (\sigma\tau_k)^{-\alpha}+\frac{1-\alpha}{\sigma}\hat{c}^{(k)}_{k-1}-\sum_{\ell=1}^{N_{q}} \varpi^{\ell}e^{-\theta^{\ell}\sigma\tau_k}
     \ge \frac{1-\alpha}{\sigma}\hat{c}^{(k)}_{k-1}-\varepsilon\ge 0,
    \end{align*}}
    where we use  the same technique as the proof of \eqref{eq:gjkineq} to obtain the  first inequality.
\end{proof}
We now analyze the approximation error of the  fast L2-1$_\sigma$ formula in the following lemma.
\begin{lemma}\label{lem4.3}
Given a function $u$ satisfying $|\partial_t^m u(t)|\le C_m(1+t^{\alpha-m})$ for $m=1,3$ and the nonuniform mesh $\{\tau_k\}_{k=1}^N$ in \eqref{mesh} satisfying \eqref{thm2cond}. The approximation error is defined by 
\begin{equation}\label{eq:localrk}
\begin{aligned}
  r_k \coloneqq& \frac{1}{\Gamma(1-\alpha)}\bigg(\int _0^{t_{k-1}}(t_k^*-s)^{-\alpha}\partial_s u(s)\ \ds- \int _0^{t_{k-1}}\sum_{\ell=1}^{N_{q}} \varpi^{\ell}e^{-\theta^{\ell}(t_k^*-s )}\partial_s I_2u(s)\,\ds\\
  &+ \int ^{t_k^*}_{t_{k-1}}(t_k^*-s)^{-\alpha}\partial_s [u(s)-\Pi_{1,k}u(s)]\ \ds \bigg)\qquad k\ge 1,  
\end{aligned}
\end{equation} 
where  $I_2 u=\Pi_{2,j}u$ on $(t_{j-1},t_j)$ for $j<k$, and $ \Pi_{1,j}$ and $\Pi_{2,j}$ are two standard Lagrange interpolation operators with the interpolation points:
 $  \{t_{j-1},t_j\},\  \{t_{j-1},t_j,t_{j+1}\}
$, respectively. 
Then for $k\geq 1$,
\begin{equation*}\label{eq:approxerror}
\begin{aligned}
   | r_k|\le &\frac{C}{\Gamma(1-\alpha)}\bigg\{  [\mathbf M]_{k,1}\big(t_2^\alpha/\alpha+t_2\big)+\sum_{j=2}^k([\mathbf M]_{k,j}-[\mathbf M]_{k,j-1})(1+\rho_{j+1})(1+t_{j-1}^{\alpha-3})\tau_j^3\\
  & \qquad+\varepsilon(t_{k-1}+t_{k-1}^\alpha/\alpha+2\sigma (1+t_{k-1}^{\alpha-3})\tau^3_k)\mathbf{1}_{k\ge2}\bigg\},
\end{aligned}
\end{equation*}
where $C>0$ is a constant depending only on $C_m$ for $m=1,3$ and $\rho_{k+1}=1$.
\end{lemma}
 \begin{proof}
The case of $k=1$ can be verified easily. We now consider the case of $k\geq 2$.
From \eqref{eq:localrk} and  Lemma~\ref{lem:SOE_approx},  it is not difficult to derive 
\begin{align}\label{eq:localrk1}
| r_k| 
  \le& \frac{\varepsilon}{\Gamma(1-\alpha)}\int _0^{t_{k-1}}|\partial_s u(s)|\ \ds+\frac{1}{\Gamma(1-\alpha)}\bigg|\int ^{t_k^*}_{t_{k-1}}(t_k^*-s)^{-\alpha}\partial_s [u(s)-\Pi_{1,k}u(s)]\ \ds\bigg|\\&+\frac{1}{\Gamma(1-\alpha)}\bigg|\int _0^{t_{k-1}}\sum_{\ell=1}^{N_{q}} \varpi^{\ell}e^{-\theta^{\ell}(t_k^*-s )}[\partial_s u(s)-\partial_s I_2u(s)]\,\ds\bigg|.\nonumber
\end{align}
Using $|\partial_s u(s)|\le C_1(1+s^{\alpha-1})$ gives 
$$
  \int _0^{t_{k-1}}|\partial_s u(s)|\ \ds\le    C_1(t_{k-1}+t_{k-1}^\alpha/\alpha).
$$
Similar to \cite[Lemma 4.5]{quan2022stability}, we can derive the following inequalities based on Lemma~\ref{proMkj}:
\begin{align*}
     &\bigg|\int _0^{t_{1}}\sum_{\ell=1}^{N_{q}} \varpi^{\ell}e^{-\theta^{\ell}(t_k^*-s )}[\partial_s u(s)-\partial_s I_2u(s)]\,\ds\bigg|\le C [\mathbf M]_{k,1}(t_2^\alpha/\alpha+t_2),\\
     &\bigg|\int _{t_j-1}^{t_{j}}\sum_{\ell=1}^{N_{q}} \varpi^{\ell}e^{-\theta^{\ell}(t_k^*-s )}[\partial_s u(s)-\partial_s I_2u(s)]\,\ds\bigg|\\ \le& C_3 ([\mathbf M]_{k,j}-[\mathbf M]_{k,j-1)}(1+\rho_{j+1})(1+t_{j-1}^{\alpha-3})\tau_j^3 ,\quad 2\le j\le k-1,
 \end{align*}
 and 
 \begin{align*}
      &\bigg|\int ^{t_k^*}_{t_{k-1}}(t_k^*-s)^{-\alpha}\partial_s [u(s)-\Pi_{1,k}u(s)]\ \ds\bigg|\\\le& 2 C_3([\mathbf M]_{k,k}-[\mathbf M]_{k,k-1})\sigma (1+t_{k-1}^{\alpha-3})\tau^3_k +2\varepsilon C_3\sigma (1+t_{k-1}^{\alpha-3})\tau^3_k. 
 \end{align*}
Combining the above inequalities, we can obtain the desired result.
\end{proof}
\begin{theorem}[Sharp finite time $H^1$-error estimate]\label{thm:err}
Assume that $u\in C^3((0,T_{\rm{soe}}],H^1_0(\Omega)\cap H^2(\Omega))$ is a solution to \eqref{eq:subdiffusion} in finite time $(0,T_{\rm soe}]$ and $|\partial_t^m u(t)|\le C_m(1+t^{\alpha-m})$, for $m=1,2,3$,  $0< t\le T_{\rm{soe}}$. If the nonuniform mesh $\{\tau_k\}_{k=1}^N$ in \eqref{mesh} satisfies \eqref{thm2cond} and  $\varepsilon\le \frac12 T_{\rm soe}^{-\alpha}$,  then the numerical solutions of fast L2-1$_\sigma$ scheme~\eqref{eq:sch_sub} satisfies the following  error estimate
\begin{equation*}\label{eq:globalerr}
    \begin{aligned}
  \max_{1\le k\le N}  \|e^k\|_{H^1(\Omega)}\le  &C\bigg\{(t_2^\alpha/\alpha+t_2)+\frac{1}{1-\alpha}\max_{2\le k\le N}(1+\rho_{k+1}) (1+t_{k-1}^{\alpha-3})(t_{k-1}^*)^\alpha\tau_k^3\tau_{k-1}^{-\alpha} \\
  &+(\tau_1^\alpha/\alpha+\tau_1)\tau_1^{\alpha/2}+\sqrt{\Gamma(1-\alpha)}\max_{2\le k\le N}(t_k^*)^{\alpha/2}(1+ t_{k-1}^{\alpha-2})\tau_k^2\\
  &+\varepsilon\max_{2\le k\le N} (t_k^*)^\alpha(t_{k-1}+t_{k-1}^\alpha/\alpha+2\sigma (1+t_{k-1}^{\alpha-3})\tau^3_k)\bigg\},
\end{aligned}
\end{equation*}
where $e^k\coloneqq
u(t_k)-u^k$ and  $C$ is a constant depending only on $C_m$, $m=1,2,3$ and $\Omega$. Moreover, for the graded mesh with grading parameter $r\ge 1$, i.e.,
\begin{equation}\label{eq:tauj}
\begin{aligned}
    t_j  = \left(\frac{j}{N}\right)^r  T_{\rm{soe}}, \quad\tau_j  = t_j -t_{j-1} = \left[\left(\frac {j} { N}\right)^r-\left(\frac {j-1}{ N}\right)^r\right]  T_{\rm{soe}},
\end{aligned}
\end{equation}
we have 
\begin{equation}\label{eq:error_graded}
   \max_{1\leq k\leq N} \|e^k\|_{H^1(\Omega)}\le \tilde C \left(N^{-\min\{r\alpha,2\}}+\varepsilon  \right).
\end{equation}
where $\tilde C$ depends on  $\alpha, $ $\Omega$,  $T_{\rm{soe}}$, and $C_m$ with $m=1,2,3$.
\end{theorem}
\begin{proof} Based on Lemma \ref{lem4.3}, the proof is similar to \cite[Theorem  4.6]{quan2022stability} except replacing the test function $e_k^{*}=(1-\alpha/2)e^k+\alpha/2 e^{k-1}$ with $-\Delta e_k^{*}$.  We omit it here and leave it for readers.
\end{proof}
\begin{remark}
The sharp finite time $H^1$-error estimate  of the fast L2-1$_\sigma$ scheme for subdiffusion equations has also been well studied in \cite{li2021second,liu2022unconditionally}.
In Theorem~\ref{thm:err}, we reduce the restriction $\rho_k\geq 2/3$ in \cite{li2021second,liu2022unconditionally} to $\rho_k\geq \eta \approx0.475329$. 
\end{remark}

\section{Numerical results}\label{sect6}
In this section, we provide some numerical results for  the fast L2-1$_\sigma$ schemes \eqref{eq:sch_sub} and \eqref{eq:sch_subnonlinear}. As in \cite{liao2018second,chen2019error}, all the discrete coefficients $a_j^{(k)}$, $c_j^{(k)}$,  $\hat{a}^{k,\ell}_{k-1}$ and $\hat{c}^{k,\ell}_{k-1}$ in \eqref{eq:akjl21sigma} and \eqref{eq:akjfastl21sigma} are computed by adaptive Gauss–Kronrod quadrature, to avoid roundoff error problems. 
\begin{exmp}\label{exm1}
  Consider the subdiffusion equation
  \eqref{eq:subdiffusion} with
 $f(t,x,y)=\Gamma(1+\alpha)(x^2-1)(y^2-1)-2t^\alpha(x^2+y^2-2)$, thus the exact solution $u(t,x,y)=t^\alpha(x^2-1)(y^2-1)$ in $\Omega=[-1,1]^2$.
\end{exmp}
 In this example, we use the graded mesh \eqref{eq:tauj}.
 We first compare the  computational cost of the fast L2-1$_\sigma$ scheme~\eqref{eq:sch_sub} with the standard L2-1$_\sigma$  scheme in  \cite{alikhanov2015new}. Here we use  spectral collocation method in space with $10^2$ Chebyshev--Gauss--Lobatto points.  We plot the CPU time w.r.t. the total number of time steps $N$ for both schemes with $\alpha=0.5$ on the left-hand side of Figure \ref{fig:cputime}.  Here we set  $T_{\rm{soe}}=1$, $\varepsilon=1e-12, \Delta t=\sigma \tau_2$ in Lemma~\ref{lem:SOE_approx} and the grading parameter $r=1.5$ in \eqref{eq:tauj}.  We observe that the CPU time of the fast L2-1$_\sigma$ scheme increases linearly w.r.t. $N$, while the cost of the original L2-1$_\sigma$  scheme increases quadratically.  
 
 To better understand the relationship between the number of quadrature nodes $N_q$ and the parameters $(\varepsilon, \Delta t)$ in Lemma~\ref{lem:SOE_approx},  we plot  $N_q$ w.r.t.  $\varepsilon$ for fixed $\Delta t=1e-5$  and $N_q$  w.r.t. $\Delta t$ for fixed $\varepsilon=1e-13$, respectively in the middle and right subfigures of Figure~\ref{fig:cputime} where $T_{\rm{soe}}=1$. 
 We can see that  $N_q$ increases almost linearly w.r.t.  both $\log(\varepsilon)$ and $\log(\Delta t)$.
 \begin{figure}[!ht]    
\centering
    \includegraphics[trim={2in 0in 1in 0in},clip,width=1\textwidth]{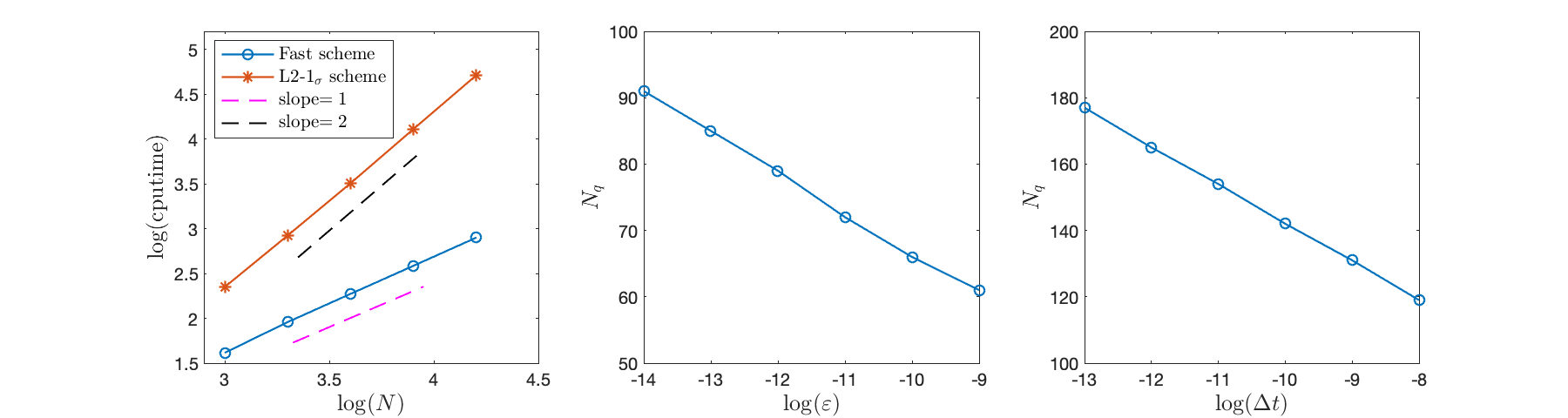}
    \vspace{-0.3in}
    \caption{(Example~\ref{exm1})  Left: CPU time in seconds w.r.t. number of time steps $N$  with  $T_{\rm{soe}}=1$, $ \varepsilon=1e-12, \Delta t=\sigma \tau_2$ for the SOE approximation  in Lemma~\ref{lem:SOE_approx} and $\alpha=0.5$, $r=1.5$ for the graded mesh \eqref{eq:tauj}. Middle:  $N_q$ w.r.t. $\varepsilon$ with $T_{\rm{soe}}=1,\ \Delta t= 1e-5$. Right:  $N_q$ w.r.t. $\Delta t$  with $T_{\rm{soe}}=1,\ \varepsilon = 1e-13$. }
    \label{fig:cputime}
\end{figure}
\begin{table}[!]
\renewcommand\arraystretch{1.1}
\begin{center}
\def\temptablewidth{1\textwidth}
\caption{(Example~\ref{exm1})  $\max_{1\leq k\leq N} \|u(t_k)-u^k\|$ for graded meshes with different grading parameters $r$ and time step numbers $N$ where $\alpha=0.3$.}\vspace{-0.2in}\label{tab1}
{\rule{\temptablewidth}{1pt}}
\begin{tabular*}{\temptablewidth}{@{\extracolsep{\fill}}cccccc}
   &$N=100$&$N=200 $ & $N=400 $&$N=800$&$N=1600$\\ \hline
$r=1/\alpha$  &1.8477e-3&	9.5337e-4&	4.8429e-4&	2.4407e-4&1.2252e-4\\
order& -- &0.9546   & 0.9772   & 0.9885   & 0.9943\\
\hline 
$r=2/\alpha$&5.2619e-5&	1.3355e-5&	3.3696e-6&	8.4709e-7&	2.1248e-7\\
order & -- &1.9782   & 1.9868 &   1.9920   & 1.9951
\\
\hline 
$r=3/\alpha$&1.1576e-4&	2.9563e-5&	7.4851e-6&	1.8855e-6&	4.7355e-7\\
order &-- &1.9693   & 1.9817  &  1.9890 &   1.9934
\\
\end{tabular*}
{\rule{\temptablewidth}{1pt}}
\end{center}
\end{table}
\begin{table}[!]
\renewcommand\arraystretch{1.1}
\begin{center}
\def\temptablewidth{1\textwidth}
\caption{(Example~\ref{exm1})  $\max_{1\leq k\leq N} \|u(t_k)-u^k\|$ for graded meshes with different grading parameters $r$ and time step numbers $N$ where $\alpha=0.5$.}\vspace{-0.2in}\label{tab2}
{\rule{\temptablewidth}{1pt}}
\begin{tabular*}{\temptablewidth}{@{\extracolsep{\fill}}cccccc}
   &$N=100$&$N=200 $ & $N=400 $&$N=800$&$N=1600$\\ \hline
$r=1/\alpha$  &2.1627e-3&	1.1124e-3&5.6416e-4	&2.8409e-4	&1.4255e-4 \\
order& -- &0.9592    &0.9795  &  0.9897&    0.9949\\
\hline 
$r=2/\alpha$&3.4309e-5&	8.6705e-6&	2.1836e-6&	5.4867e-7&	1.3764e-7\\
order & -- &1.9844   & 1.9894   & 1.9927   & 1.9950
\\
\hline 
$r=3/\alpha$&7.5018e-5&	1.9013e-5&	4.7967e-6&	1.2065e-6&	3.0290e-07\\
order &-- &1.9802  &  1.9869  &  1.9912  &  1.9940
\\
\end{tabular*}
{\rule{\temptablewidth}{1pt}}
\end{center}
\end{table}
\begin{table}[!]
\renewcommand\arraystretch{1.1}
\begin{center}
\def\temptablewidth{1\textwidth}
\caption{(Example~\ref{exm1}) $\max_{1\leq k\leq N} \|u(t_k)-u^k\|$ for graded meshes with different grading parameters $r$ and time step numbers $N$ where $\alpha=0.7$.}\vspace{-0.2in}\label{tab3}
{\rule{\temptablewidth}{1pt}}
\begin{tabular*}{\temptablewidth}{@{\extracolsep{\fill}}cccccc}
   &$N=100$&$N=200 $ & $N=400 $&$N=800$&$N=1600$\\ \hline
$r=1/\alpha$  &1.8710e-3&	9.5944e-4&	4.8584e-4&	2.4447e-4&	1.2262e-4 \\
order& -- & 0.9636   &0.9817  &  0.9908   & 0.9954\\
\hline 
$r=2/\alpha$&2.0342e-5&	5.1591e-6&	1.2961e-6&	3.2443e-7&	8.1133e-8\\
order & -- & 1.9793  &  1.9929 &   1.9982   & 1.9996
\\
\hline 
$r=3/\alpha$&3.9451e-5&	9.9647e-6&	2.5108e-6&	6.3158e-07&	1.5866e-7\\
order &-- &1.9852  &  1.9887   & 1.9911  &  1.9930
\\
\end{tabular*}
{\rule{\temptablewidth}{1pt}}
\end{center}
\end{table}
 \begin{figure}[!ht]    
\centering
    \includegraphics[trim={2in 0in 1.2in 0in},clip,width=1.\textwidth]{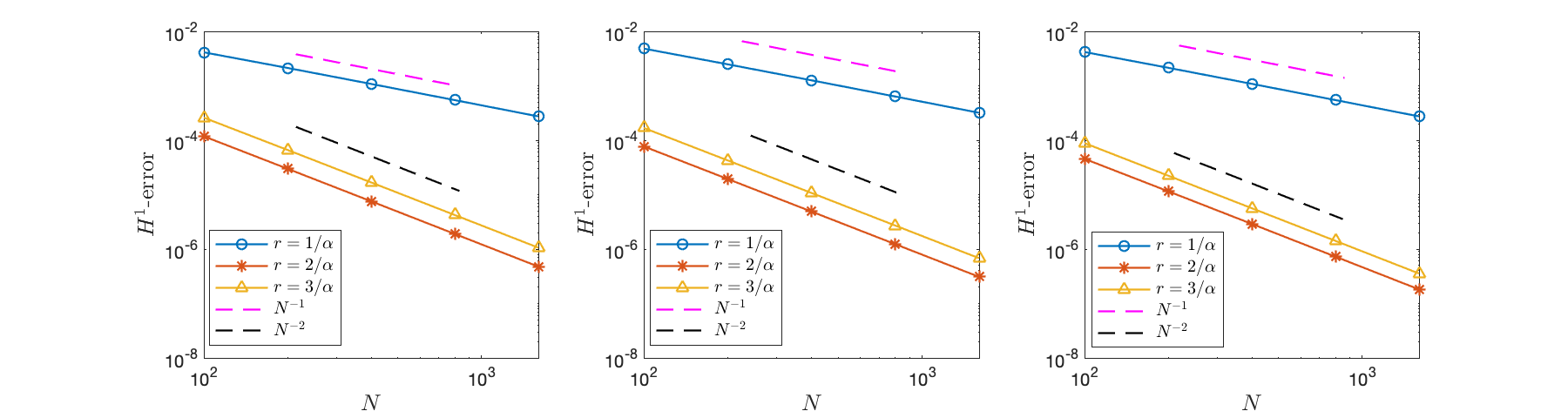}
   \vspace{-0.3in}
    \caption{(Example \ref{exm1}) Maximum $H^1$-error w.r.t. $N$ for $\alpha=0.3,\ 0.5,\ 0.7$ from left to right  with  different $r$ for graded meshes \eqref{eq:tauj}. }
    \label{fig:H1error}
\end{figure}

 We now test the convergence of the fast L2-1$_\sigma$ scheme  \eqref{eq:sch_sub}.  Let $T_{\rm{soe}}=1$, $\varepsilon=1e-12$, $\Delta t=\sigma\tau_2$ in the Lemma~\ref{lem:SOE_approx}, and apply the spectral collocation method in the space with $25^2$ Chebyshev--Gauss--Lobatto points.  Table~\ref{tab1}--\ref{tab3} present the  maximum $L^2$-errors with $\alpha=0.3,\ 0.5,\ 0.7$ and $r = 1/\alpha,\ 2/\alpha,\ 3/\alpha$ for graded meshes respectively. It can be observed that the convergence rates in $L^2$-norm are $N^{-\min\{r\alpha,2\}}$ for the fast L2-1$_\sigma$ scheme~\eqref{eq:sch_sub}  on graded meshes.
 The  maximum $L^2$-error is computed by $\max_{1\leq k\leq N} \|u(t_k)-u^k\|$.  We  show the  maximum $H^1$-error in Figure \ref{fig:H1error} with $\alpha=0.3,\ 0.5,\ 0.7$ and $r = 1/\alpha,\ 2/\alpha,\ 3/\alpha$ for graded meshes respectively. The  maximum  $H^1$-error is computed by $\max_{1\leq k\leq N} \|\nabla(u(t_k)-u^k)\|$.  We can also find that the  convergence rates in $H^1$-norm for the fast L2-1$_\sigma$ scheme~\eqref{eq:sch_sub} on graded meshes  are $N^{-\min\{r\alpha,2\}}$, which is consistent with Theorem~\ref{thm:err}.
\begin{exmp}[Long time simulation]\label{exm:global}
Consider the subdiffusion equation
  \eqref{eq:subdiffusion}, with two different sources terms
$  
      f_1(t,x,y)=t\sin(0.2t),~ f_2(t,x,y)=5\exp(-0.0005t)\sin(0.005t).
$
The initial condition is set to $u^0=\sin(\pi x)\sin(\pi y)$ in  $\Omega=[-1,1]^2$.
\end{exmp}

We test the global-in-time $H^1$-stability of the fast L2-1$_\sigma$ scheme \eqref{eq:sch_sub} with source term $f_1$ and $f_2$ respectively. The  spectral collocation method is used in space with $10^2$ Chebyshev--Gauss--Lobatto points. 
We set $T_{\rm{soe}}=1$, $\varepsilon=1e-12$, $\Delta t=\sigma\tau_2$ in  Lemma~\ref{lem:SOE_approx}.  
In addition, we use the following nonuniform mesh:
\begin{equation*}
\tau_j = \left\{
\begin{aligned}
& \left(\frac {j} {100}\right)^r-\left(\frac {j-1}{100}\right)^r && j \leq 100\\
    &1.005 \tau_{j-1} &&  j>100, ~1.005\tau_{j-1}<\tau_{\max} \coloneqq 0.2, \\
    &\tau_{\max} && \mbox{otherwise,}
\end{aligned}
\right.
\end{equation*}
where $r = 2/\alpha$.

We plot $\|\nabla u\|$ w.r.t. time $t$ in Figure~\ref{fig:H1nostability} and Figure~\ref{fig:H1stability} for $f_1$ and $f_2$ respectively, computed by the fast L2-1$_\sigma$ scheme~\eqref{eq:sch_sub} with $\alpha=0.5$.  
Note that the time $t_n$ can be much larger than $T_{\rm{soe}}$.
We can see that  $\|\nabla u\|$ is bounded in the case of $f_2$, but not in the case of $f_1$.  
One explanation is that  $f_1 \notin L^\infty([0,\infty);L^2(\Omega)) \cap BV([0,\infty); L^2(\Omega))$ so that the assumption on source term in Theorem~\ref{thm2} is not satisfied, while  $f_2\in L^\infty([0,\infty);L^2(\Omega)) \cap BV([0,\infty); L^2(\Omega))$ holds true.
   \begin{figure}
    \centering
    \includegraphics[trim={1.3in 0in 1in 0in},clip,width=1.\textwidth]{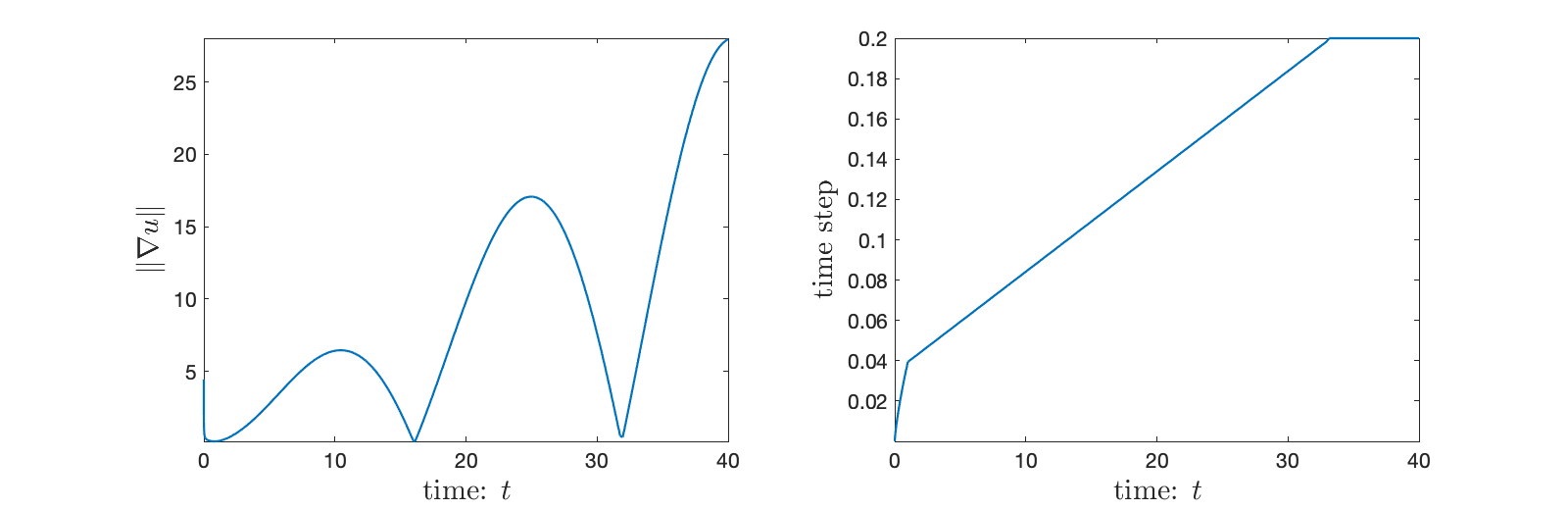}
    \vspace{-0.3in}
     \caption{(Example~\ref{exm:global}) Long time $H^1$-instability. Left: $\|\nabla u\|$ w.r.t. time $t$. Right:  Time step w.r.t. time $t$.}
    \label{fig:H1nostability}
\end{figure}
  \begin{figure}
    \centering
    \includegraphics[trim={1.3in 0in 1in 0in},clip,width=1.\textwidth]{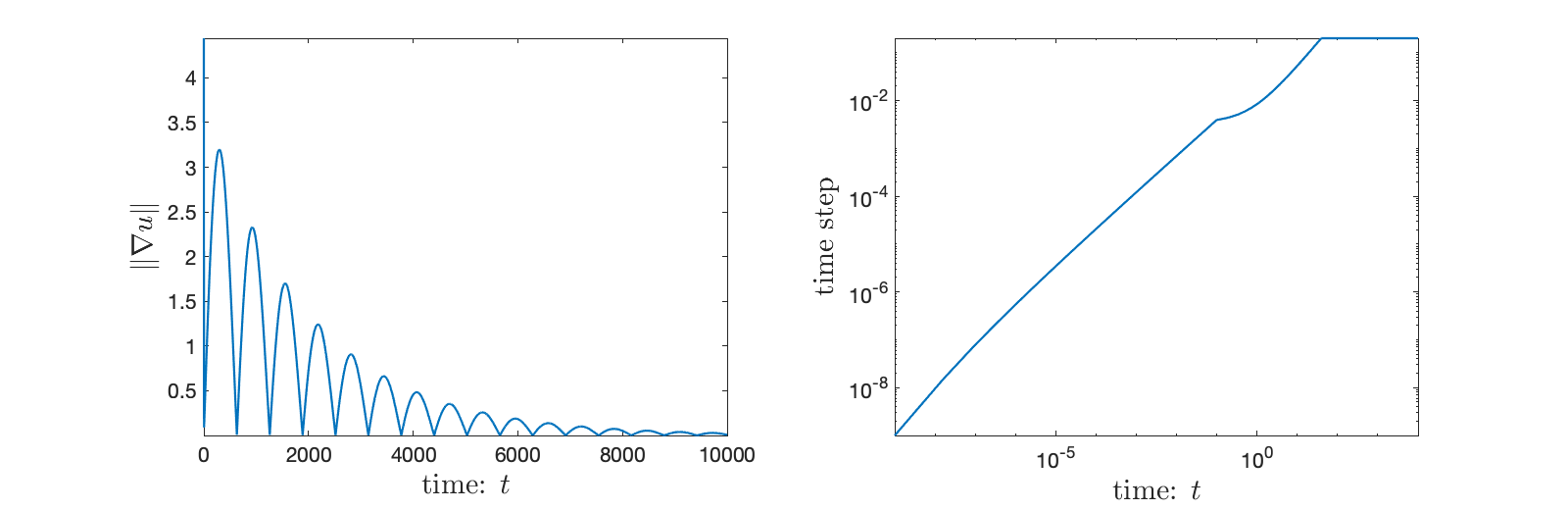}
    \vspace{-0.3in}
     \caption{(Example~\ref{exm:global}) Long time $H^1$-stability. Left: $\|\nabla u\|$ w.r.t. time: $t$.  Right:  Time step w.r.t. time: $t$.}
    \label{fig:H1stability}
\end{figure}
\begin{exmp}\label{eq:nonexpl}
\begin{enumerate}
    \item[(a)] Consider the semilinear subdiffusion equation
  \begin{equation*} 
    \begin{aligned}
        \partial_t^\alpha u = \Delta u +\sin(u)+ g(t,x,y),\quad (t,x,y)\in (0,1]\times(-1,1)^2
    \end{aligned}
\end{equation*}
with  homogeneous Dirichlet boundary condition.
To investigate the convergence orders in time, the initial condition $u^0$ and source term $g$ are chosen such that the exact solution is
 $u(t,x,y)=t^\alpha \sin(\pi x)\sin(\pi y)$.
  \item[(b)]  Consider the semilinear subdiffusion equation
  \begin{equation*} 
    \begin{aligned}
        \partial_t^\alpha u  =\nu^2 \Delta u +\sin(u ),\quad   (t,x,y)\in (0,\infty)\times(-1,1)^2
    \end{aligned}
\end{equation*}
 with homogeneous Dirichlet boundary condition, $\nu^2=0.01$, and  initial condition 
 $u^0=\sin(\pi x)\sin(\pi y)$.
 \end{enumerate}
\end{exmp}
For the problem (a), we use  the graded meshes \eqref{eq:tauj}.  
We set  $T_{\rm{soe}}=1$, $\varepsilon=1e-12$, $\Delta t=\sigma\tau_2$ in  Lemma~\ref{lem:SOE_approx}, and use spectral collocation method in space with $25^2$ Chebyshev--Gauss--Lobatto points. Figure \ref{fig:L2H1errornonlinear} shows the  maximum $L_2$-errors and  maximum $H^1$-error for $\alpha=0.3,\ 0.5,\ 0.7$ with $r = 1/\alpha,\ 2/\alpha,\ 3/\alpha$ for graded meshes~\eqref{eq:tauj}.
Again, it can be observed that  the convergence rates in both $L^2$-norm and $H^1$-norm are $N^{-\min\{r\alpha,2\}}$ for graded meshes, which agree with the results in \cite{liu2022unconditionally}.
\begin{figure}[!ht]    
\centering
    \includegraphics[trim={2in 0in 2in 0in},clip,width=1.\textwidth]{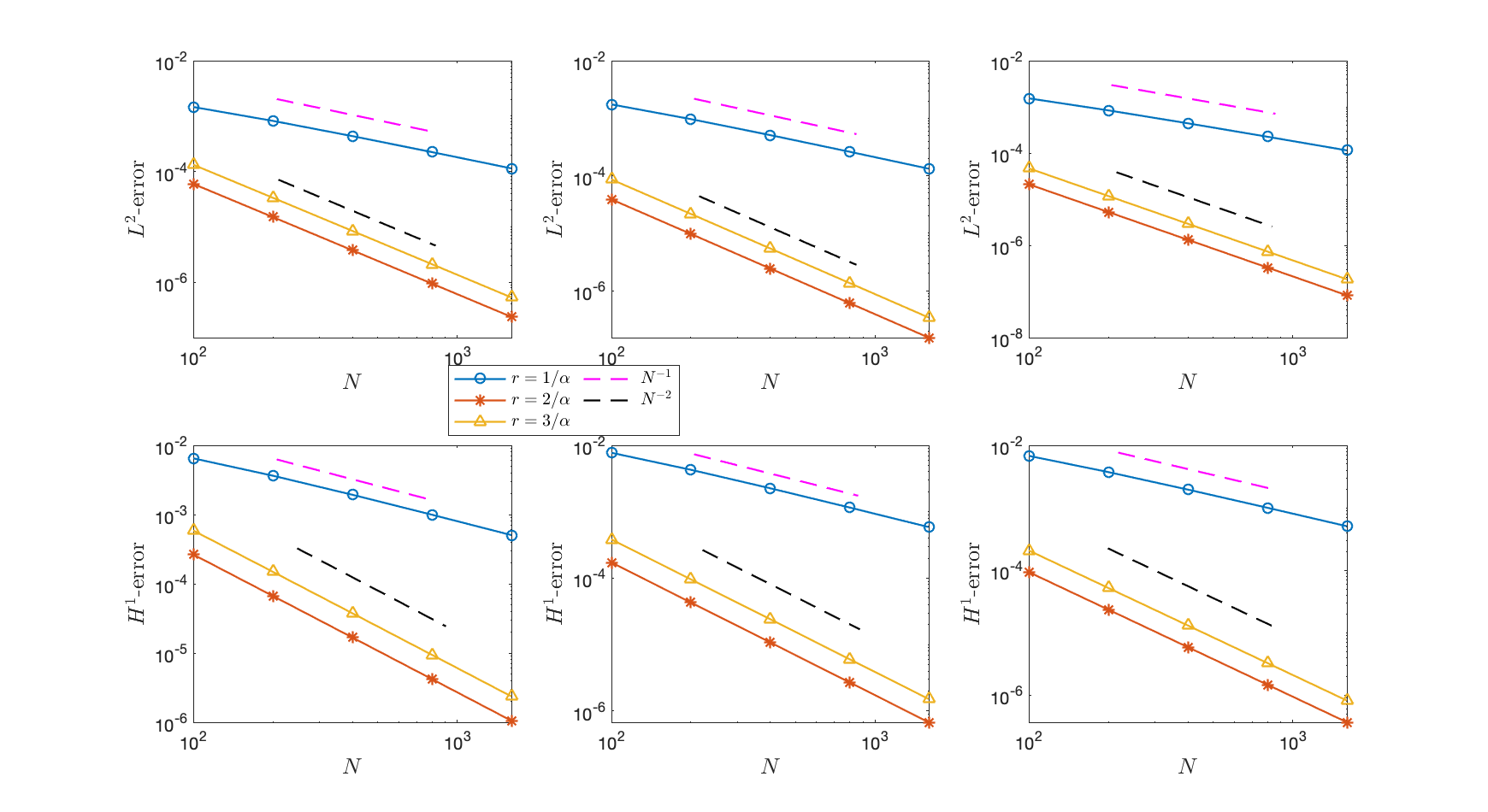}
    \vspace{-0.3in}
    \caption{(Example \ref{eq:nonexpl}) Maximum $L^2$-error (top) and $H^1$-error (bottom) w.r.t. $N$ for $\alpha=0.3,\ 0.5,\  0.7$ from  left to  right,  with  different $r$ for graded mesh \eqref{eq:tauj}. }
    \label{fig:L2H1errornonlinear}
\end{figure}

We now turn to show the  energy evolution for the problem (b) of Example \ref{eq:nonexpl} for $\alpha=0.3,\ 0.5,\ 0.7$. The energy is $
     E=\int_{\Omega}\left(\frac{\nu^2}{2}|\nabla u|^2+\cos(u)\right)\ \dx.
$
The spectral collocation method is applied in space with $10^2$ Chebyshev--Gauss--Lobatto points.
Let $T_{\rm{soe}}=300$, $\varepsilon=1e-12$, $\Delta t=\sigma\tau_2$ in  Lemma~\ref{lem:SOE_approx}. 
We use the following nonuniform time mesh:
\begin{equation}\label{eq:mesh}
\tau_j = \left\{
\begin{aligned}
& 0.1\left[\left(\frac {j} {100}\right)^r-\left(\frac {j-1}{100}\right)^r\right] && j \leq 100\\
    &1.01 \tau_{j-1} &&  j>100, ~1.01\tau_{j-1}<\tau_{\max} \coloneqq 0.02, \\
    &\tau_{\max} && \mbox{otherwise,}
\end{aligned}
\right.
\end{equation}
where $r =3$. This mesh satisfies  all constraints in Theorem~\ref{thm:nonlinear}. The energy evolution and time steps are illustrated in Figure~\ref{fig:energy1stability}, which is consistent with Theorem~\ref{thm:nonlinear}.

\begin{exmp}\label{eq:nonexamAC}
Consider the time-fractional Allen--Cahn (TFAC) equation 
  \begin{equation*} \label{eq:nonlinear2}
    \begin{aligned}
        \partial_t^\alpha u =\nu^2 \Delta u-f(u),\quad  (t,x,y)\in (0,\infty)\times(-1,1)^2
    \end{aligned}
\end{equation*}
 with  homogeneous Dirichlet boundary condition, $\nu^2=0.01$, 
 $u^0=0.05(2*\text{rand}(x,y)-1)$, and $f(u) = u^3-u$, where rand$(\cdot)$ stands for the random values in $(0,1)$.
\end{exmp}

Note that the global Lipschitz property  fails for the TFAC equation. Fortunately, it is well studied in \cite{du2020time} that  the TFAC equation satisfies the maximum principle, i.e., $\|u\|_\infty \leq 1$ if $\|u^0\|_\infty \leq 1$.
We use the following truncation technique for $f(u)$ (see for example \cite{shen2010numerical})
\begin{equation*} \label{eq:gk}
\hat{f}(u)= \\
\left\{
\begin{aligned}
&(3M^2-1)u-2M^3, &&u> M,\\
&u^3-u,&&|u|\le M,\\
    &(3M^2-1)u+2M^3, &&u<-M,
\end{aligned}
\right.
\end{equation*}  
one primitive integral of which is 
 \begin{equation*}
    \hat{F}(u)= \\
\left\{
\begin{aligned}
&\frac{3M^2-1}{2}u^2-2M^3u+\frac{3M^4+1}{4}, &&u> M,\\
      &\frac14(u^2-1)^2,&&|u|\le M,\\
    &\frac{3M^2-1}{2}u^2+2M^3u+\frac{3M^4+1}{4}, &&u<-M.
\end{aligned}
\right.
\end{equation*}
Here $M\geq 1$ is some positive constant. In particular, we set $M=1$ so that $\sup_{u\in\mathbb R} |\hat{f}'(u)|\leq 2 $, i.e., the Lipschitz condition is satisfied. 
Let $T_{\rm{soe}}=300$, $\varepsilon=1e-12$, $\Delta t=\sigma\tau_2$ in  Lemma~\ref{lem:SOE_approx}. 
The space is discretized by the spectral collocation method with $10^2$ Chebyshev--Gauss--Lobatto points. 
For the time discretization, we still use the time mesh~\eqref{eq:mesh}.

 Firstly we plot the energy evolution for the fast L2-1$_\sigma$ scheme~\eqref{eq:sch_subnonlinear} using the aforementioned truncation technique with $\alpha=0.4,\ 0.6,\ 0.8$.
 In this case,  the energy is computed by 
$
     \hat E=\int_{\Omega}\left(\frac{\nu^2}{2}|\nabla u|^2+\hat{F}(u)\right)\ \dx.
$ 
 It can be verified that the time mesh \eqref{eq:mesh} satisfies all the conditions in Theorem~\ref{thm:nonlinear} for $\alpha=0.6,\  0.8$. The energy stability result can be observed on the left-hand side of  Figure~\ref{fig:energy2stability} which agrees with  Theorem~\ref{thm:nonlinear}. However, the energy stability is still  observed for $\alpha=0.4$ despite that the conditions in Theorem~\ref{thm:nonlinear} are not satisfied.
 
 Secondly, we compare the original energies of the fast L2-1$_\sigma$ schemes~\eqref{eq:sch_subnonlinear} with $f(u)$ and $\hat{f}(u)$, on the right-hand side of Figure~\ref{fig:energy2stability}, where we take the time mesh~\eqref{eq:mesh} and $\alpha=0.8$. Here  the original energy is calculated by 
$
   E= \int_{\Omega}\left(\frac{\nu^2}{2}|\nabla u|^2+\frac14 (u^2-1)^2\right)\dx.
$
It is observed that the corresponding original energies almost coincide and both of them decay monotonically in time, although we can not provide a rigorous proof of this energy dissipation.
   \begin{figure}
    \centering
    \includegraphics[trim={1.3in 0in 1in 0in},clip,width=1.\textwidth]{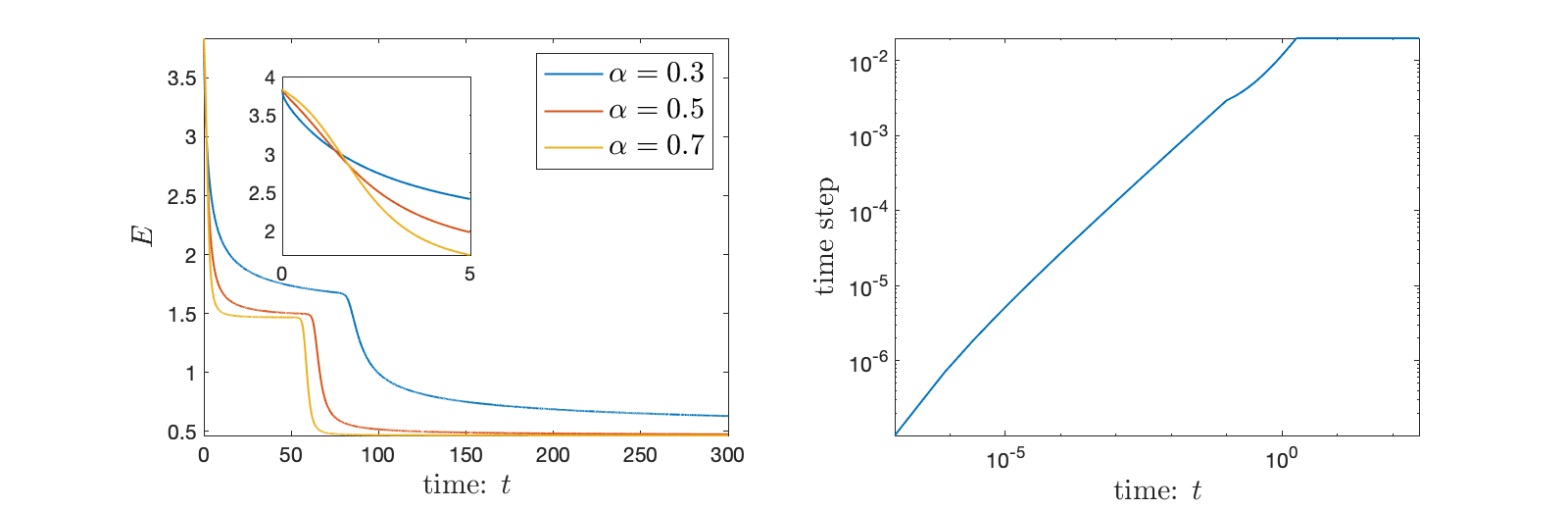}
    \vspace{-0.3in}
    \caption{(Example~\ref{eq:nonexpl}) Left: Energy evolution  w.r.t. time $t$ with $\alpha=0.3,\ 0.5,\ 0.7$. Right: Time step w.r.t.  time $t$.}
    \label{fig:energy1stability}
\end{figure}
  \begin{figure}
    \centering
    \includegraphics[trim={1.3in 0in 1in 0in},clip,width=1.\textwidth]{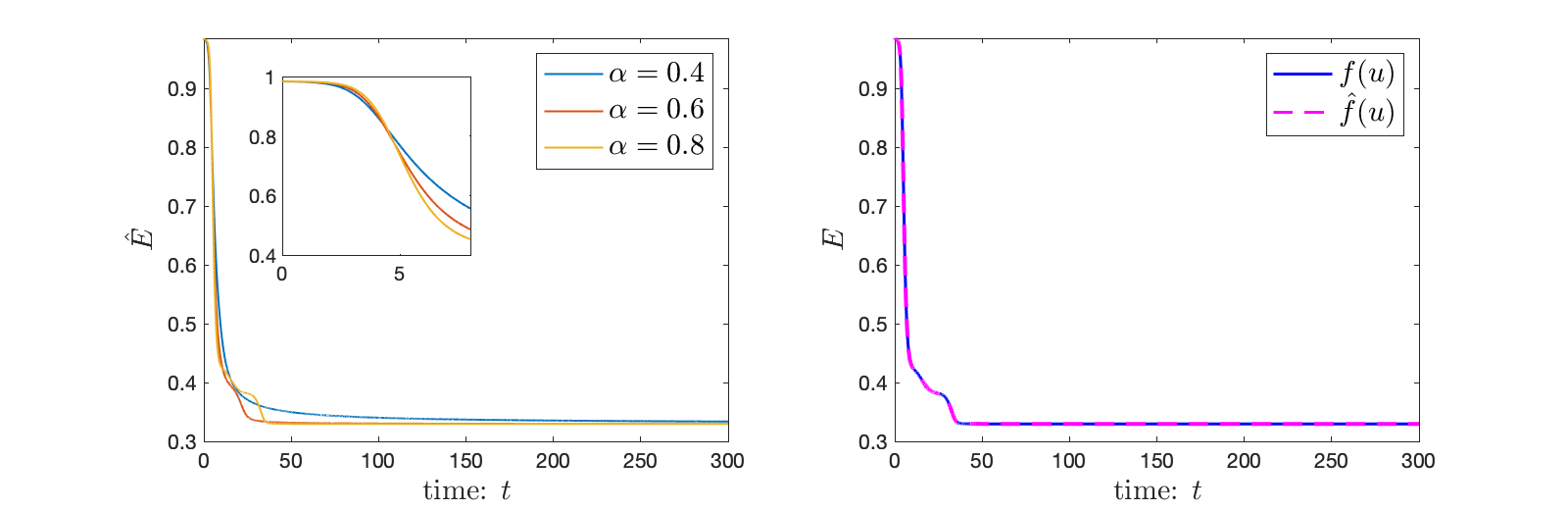}
    \vspace{-0.3in}
    \caption{ (Example~\ref{eq:nonexamAC}) Left: Energy evolution  w.r.t. time $t$ for $\alpha=0.4,\  0.6, \ 0.8$. Right:   Comparison of original energies  w.r.t. time $t$ with $\alpha=0.8$ for $f(u)$ and $\hat{f}(u)$.}
    \label{fig:energy2stability}
\end{figure}

\section{Conclusion}\label{sect7}
We have proved a new global-in-time $H^1$-stability of the fast L2-1$_\sigma$ scheme for linear and semilinear subdiffusion equations. 
The proposed scheme is constructed by combining  the L2-1$_\sigma$ scheme with the SOE approximation to the convolution kernel involved in the Caputo
fractional derivative.  
A crucial bilinear form $\mathcal B_n$, associated with the fast L2-1$_\sigma$ formula,  is proved to be positive semidefinite under some mild constraints of time steps for fixed SOE parameters $\varepsilon$, $\Delta t$, and $T_{\rm{soe}}$. 
Note that this positive semidefiniteness holds not only when $t\leq T_{\rm{soe}}$ but also for any time $t>T_{\rm{soe}}$.
Therefore, the uniform long time $H^1$-stability still holds as the time goes to infinity.

We also revisit the sharp $H^1$-error estimate in finite time  for the fast L2-1$_\sigma$ scheme, and we relax the constraint on the stepsize ratios from $\rho_k\geq 2/3$ in \cite{li2021second,liu2022unconditionally} to $\rho_k\geq \eta \approx0.475329$. So far, the long time optimal $H^1$-error estimate is still open for  the fast L2-1$_\sigma$ scheme with general nonuniform meshes. This might be solved with the help of positive semidefiniteness of the bilinear form associated with $F_k^{\alpha,*}$.

\section*{Acknowledgements}
C. Quan is supported by NSFC Grant 12271241, the Stable Support Plan Program of Shenzhen Natural Science Fund (Program Contract No. 20200925160747003), and Shenzhen Science and Technology Program (Grant No. RCYX20210609104358076). J. Yang is supported by the National Natural Science Foundation of China (NSFC) Grant No. 12271240, the NSFC/Hong Kong RGC Joint Research Scheme (NSFC/RGC 11961160718), the fund of the Guangdong Provincial Key Laboratory of Computational Science and Material Design, China (No.2019B030301001), and the Shenzhen Natural Science Fund (RCJC20210609103819018).

\bibliography{bibfile}
\bibliographystyle{amsplain}

\end{document}